\documentclass[11pt,a4paper]{amsart}
\usepackage{amssymb, amsmath, amsthm, amsbsy, amscd, mathrsfs,stmaryrd}

\usepackage{pdflscape}

\usepackage[draft]{graphicx}
\usepackage[vcentermath]{youngtab}

\usepackage{enumitem}
\usepackage{extarrows}
\usepackage[mathcal]{eucal}

\usepackage{array,float,hyperref,yhmath}

\usepackage[usenames]{color}

\usepackage[usenames]{color}

\hypersetup{
    colorlinks=true,
    linkcolor=blue,
    filecolor=cyan,
    urlcolor=blue,
    citecolor=magenta
}

\input xy
\xyoption{all}

\numberwithin{equation}{section}
\numberwithin{figure}{section}

\theoremstyle{plain}
\newtheorem{theorem}{Theorem}[section]
\newtheorem{proposition}[theorem]{Proposition}
\newtheorem{lemma}[theorem]{Lemma}
\newtheorem{corollary}[theorem]{Corollary}

\newtheorem*{conjecture*}{Conjecture}
\newtheorem{thmABC}{Theorem}

\newtheorem{corABC}[thmABC]{Corollary}

\theoremstyle{definition}
\newtheorem{definition}[theorem]{Definition}

\theoremstyle{remark}

\newtheorem{remark}[theorem]{Remark}

\newcommand{\be}{\begin{enumerate}}
        \newcommand{\ee}{\end{enumerate}}

\newcommand{\brem}{\begin{remark}}
        \newcommand{\erem}{\end{remark}}

\newcommand{\bp}{\begin{proof}}
        \newcommand{\ep}{\end{proof}}

\newcommand{\N}{\ensuremath{\mathbb{N}}}

\newcommand{\C}{\ensuremath{\mathbb{C}}}
\newcommand{\R}{\ensuremath{\mathbb{R}}}

\newcommand{\bfc}{\ensuremath{\mathbf{c}}}

\newcommand{\bfG}{\ensuremath{\mathbf{G}}}
\newcommand{\bfJ}{\ensuremath{\mathbf{J}}}
\newcommand{\bfH}{\ensuremath{\mathbf{H}}}

\newcommand{\bfU}{\ensuremath{\mathbf{U}}}

\newcommand{\mcB}{\mathcal{B}}

\newcommand{\mcM}{\mathcal{M}}

\newcommand{\mfg}{\ensuremath{\mathfrak{g}}}
\newcommand{\mfh}{\ensuremath{\mathfrak{h}}}
\newcommand{\mfj}{\ensuremath{\mathfrak{j}}}
\newcommand{\mfjL}{\ensuremath{\mathfrak{j}_L}}

\newcommand{\mfs}{\ensuremath{\mathfrak{s}}}

\newcommand{\gl}{\mathfrak{gl}}
\newcommand{\gu}{\mathfrak{gu}}

\renewcommand{\sl}{\mathfrak{sl}}
\newcommand{\su}{\mathfrak{su}}

\newcommand{\rf}{\ensuremath{{\bf k}}}
\newcommand{\Rf}{\ensuremath{{\bf K}}}
\newcommand{\lri}{\mathfrak o}
\newcommand{\cO}{\lri}

\newcommand{\ol}{\ensuremath{\overline}}

\newcommand{\wt}{\ensuremath{\widetilde}}

\newcommand{\tr}{\textup{tr}}
\newcommand{\Tr}{\mathsf{T}}

\newcommand{\Lri}{\mathfrak O}

\newcommand{\Gri}{\ensuremath{\mathcal{O}}}

\newcommand{\Gf}{\ensuremath{F}}

\newcommand{\mfp}{\mathfrak{p}}
\newcommand{\mfP}{\mathfrak{P}}

\newcommand{\wh}{\widehat}

\newcommand{\kk}{\ensuremath{\mathbf{k}}}
\newcommand{\KK}{\ensuremath{\mathbf{K}}}

\newcommand{\udots}{\mathinner{\mskip1mu\raise1pt\vbox{\kern7pt\hbox{.}}
                \mskip2mu\raise4pt\hbox{.}\mskip2mu\raise7pt\hbox{.}\mskip1mu}}

\DeclareMathOperator{\val}{val}

\DeclareMathOperator{\Cent}{C}

\DeclareMathOperator{\SL}{SL}
\DeclareMathOperator{\SU}{SU}

\DeclareMathOperator{\diag}{diag}

\DeclareMathOperator{\Hom}{Hom}

\DeclareMathOperator{\Id}{Id}

\DeclareMathOperator{\Stab}{Stab}

\DeclareMathOperator{\Ad}{Ad}

\DeclareMathOperator{\Mat}{Mat}

\DeclareMathOperator{\GL}{GL}
\DeclareMathOperator{\GU}{GU}
\DeclareMathOperator{\Sh}{Sh}

\DeclareMathOperator{\Irr}{Irr}

\DeclareMathOperator{\JLc}{J_{L}}

\DeclareMathOperator{\JUc}{J_{U}}

\DeclareMathOperator{\JJ}{{\bf J}}
\DeclareMathOperator{\JJc}{J}
\DeclareMathOperator{\jc}{\mfj}

\DeclareMathOperator{\Ind}{Ind}

\DeclareMathOperator{\Ker}{Ker}
\renewcommand{\epsilon}{\varepsilon}
\renewcommand{\phi}{\varphi}
\newcommand{\nonsq}{\rho}

\newcommand{\half}{\frac{1}{2}}
\newcommand{\lup}{{\ell_2}}
\newcommand{\ldown}{{\ell_1}}
\newcommand{\mup}{{m_2}}
\newcommand{\mdown}{{m_1}}
\newcommand{\upu}{{m_2}}
\newcommand{\downu}{{m_1}}

\newcommand{\iga}{\mathbf{I}_{\bfG}(\psi_{A_{\ldown}})}

\newcommand{\balpha}{\bar{\alpha}}
\newcommand{\bbeta}{\bar{\beta}}

\newcommand{\bsigma}{\bar{\sigma}}
\newcommand{\bnu}{\bar{\nu}}

\newcommand{\bX}{\bar{X}} 
\newcommand{\btau}{\bar{\tau}}

\newcommand{\mat}[9]{\begin{bmatrix}
                #1 & #2 & #3 \\
                #4 & #5 & #6 \\
                #7 & #8 & #9
\end{bmatrix} }
\newcommand{\newA}{{\left[\begin{matrix} \omega & 0 & 1 \\  \alpha & \omega + \delta & 0 \\ \gamma + \delta^2 &  \beta  & \omega  \end{matrix}\right] }}

\newcommand{\mcS}{\mathcal{S}}

\begin{document}

\title[Representation zeta functions of groups of type $A_2$]{Representation zeta functions of  groups of type $A_2$ in positive characteristic}

\date{\today} 


\author{Uri Onn}
\address{Mathematical Sciences Institute, The Australian National University, Canberra, Australia}
\email{uri.onn@anu.edu.au}

\author{Amritanshu Prasad}
\address{The Institute of Mathematical Sciences (HBNI), CIT Campus Taramani, Chennai 600113, India}
\email{amri@imsc.res.in}

\author{Pooja Singla}
\address{Department of Mathematics and Statistics, IIT Kanpur,  Kanpur 208016, India}
\email{psingla@iitk.ac.in}

\begin{abstract}  We prove two conjectures regarding the representation growth of groups of type~$A_2$. The first, conjectured by Avni, Klopsch, Onn, and Voll, regards the uniformity of representation zeta functions over local complete discrete valuation rings. The second is the Larsen--Lubotzky conjecture on the representation growth of irreducible lattices in groups of type $A_2$ in positive characteristic, assuming Serre's conjecture on the congruence subgroup problem.     
\end{abstract}
\maketitle
\setcounter{tocdepth}{1} \tableofcontents{}

\thispagestyle{empty}

   \section{Introduction}

        Let $G$ be a group and let $\Irr(G)$ denote the set of isomorphism classes of its finite-dimensional complex irreducible representations. If $G$ is a topological or a complex algebraic group, we further assume that the representations are continuous or rational, respectively.  For ${n \in \N}$ let $r_n(G)$ be the number of $n$-dimensional irreducible representations of $G$. The group~$G$ is said to have {polynomial representation growth} (PRG) if the sequence $r_n(G)$ is bounded by a polynomial function of $n$.  When $G$ has PRG, one defines its {\em representation zeta function} to be the associated Dirichlet generating function
        \[
                \zeta_G(s) =\sum_{n=1}^\infty r_n(G) n^{-s}=\sum_{\rho \in \Irr(G)} (\dim \rho)^{-s},
        \]
where $s$ is a complex variable.  The series $\zeta_G(s)$ converges for $\mathrm{Re}(s) > \alpha_G$, where
\[
\alpha_G=\inf \{ \sigma  \in \R \mid \zeta_G(\sigma) < \infty \}
\]
is the {abscissa of convergence}.

\smallskip

Representation zeta functions originated, albeit with different terminology, in the work of Frobenius~\cite{Frobenius1896a} at the end of the 19th century and later on extended by Mednyh \cite{Mednyh1978} in relation to counting points on character varieties of finite groups. In the early 1990's, they reappeared in the work of Witten \cite{Witten1991} on two-dimensional TQFT in relation to the volume of moduli spaces of flat connections on principal $G$-bundles on Riemann surfaces for~$G$ a semisimple Lie group. In the past 20 years a substantial body of work has been devoted to the representation zeta functions of $p$-adic  and arithmetic groups; see  \cite{AKOV2013,AKOV2016,JaikinZapirain2006,LL2008,LubMar2004} and references therein. See also \cite{AvniAizenbud2016} for the relation between their special values and rational singularities of the moduli space of local systems.
        \smallskip

        Our focus is on the representation zeta functions of arithmetic groups of type~$A_2$. Let~$\Gf$ be a global field, namely, a number field or the field of rational functions of an irreducible algebraic curve defined over a finite field. Let $\Gri \subset \Gf$ be the ring of integers and $\Gri_S \subset \Gf$ the ring of $S$-integers for a finite set of places $S$ of $\Gf$ that contains all the Archimedean ones. Let $\bfG$ be a connected, simply connected semisimple algebraic group defined over $\Gf$.
            By the strong approximation theorem \cite[Theorem~7.12]{MR1278263} the map 
        \[
        \eta: \wh{\bfG(\Gri_S)} \longrightarrow \bfG(\wh{\Gri_S})
        \]
        from the profinite completion to the congruence completion is surjective.
        The group $\bfG(\Gri_S)$ is said to satisfy the \emph{congruence subgroup property} (CSP) if $\eta$ has a finite kernel. 
        Lubotzky and Martin \cite{LubMar2004} proved that if the characteristic of $\Gf$ is odd and $\bfG(\Gri_S)$ satisfies the CSP then it has PRG (with a reverse implication in characteristic zero). Larsen and Lubotzky~\cite{LL2008} then obtained an Euler product decomposition for the representation zeta function of $\bfG(\Gri_S)$ in the case where $\ker\eta$ is trivial:
        \begin{equation}
          \label{eq:euler}
          \zeta_{\bfG(\Gri_S)} (s) = \zeta_{\bfG(\mathbb C)}(s)^r\times \prod_{v\notin S} \zeta_{\bfG(\Gri_v)}(s),
        \end{equation}
        where $r=0$ when $\Gf$ is a function field, and $r=[\Gf:\mathbb Q]$ when $\Gf$ is a number field.

        The following theorem was proved by Avni, Klopsch, Onn, and Voll \cite{AKOV2} in the number field case. In this paper, we extend it to the function field case.
        
        \begin{thmABC}\label{thm:basechange} Let $\Gf$ be a global field of characteristic zero or greater than $3$, and let $\Gri_{S}$ be the ring of $S$-integers in $\Gf$ for some finite set of places~$S$. Let $\bfG$ be a connected, simply connected absolutely almost simple algebraic group defined over~$\Gri_S$ with absolute root system of type~$A_2$. Suppose that the group $\bfG(\Gri_S)$ has the CSP. Then $\alpha_{\bfG(\Gri_S)}=1$. Moreover, the representation zeta function $\zeta_{\bfG(\Gri_S)}(s)$ can be meromorphically continued slightly to the left of the abscissa and has a pole of order $2$ at $s=1$.
        \end{thmABC}

        The proof relies on the computation of local factors in \eqref{eq:euler}.
        These computations were carried out in the characteristic zero case in \cite{AKOV2} using the Kirillov orbit method, which is not available in positive characteristic. An important consequence of Theorem~\ref{thm:basechange} is a proof of the following conjecture of Larsen and Lubotzky assuming CSP for irreducible lattices in a semisimple group of type $A_2$ over a field of positive characteristic. The characteristic zero counterpart for groups of type $A_2$ was proved in \cite{AKOV2013} and for arbitrary groups in characteristic zero in \cite{AKOV2016}.

        \begin{conjecture*}[Larsen--Lubotzky \cite{LL2008}] Let $H$ be a higher-rank semisimple group, that is, $H=\prod_i\bfG_i(F_i)$ where each $F_i$ is a local field, each $\bfG_i$ is an absolutely almost simple $F_i$-group, and $\sum_i \mathrm{rank}_{F_i} (\bfG_i ) \ge 2$. Then for any two irreducible lattices $\Gamma_1,\Gamma_2 \leq H$, $\alpha_{\Gamma_1}=\alpha_{\Gamma_2}$.
        \end{conjecture*}

This conjecture can be regarded as a quantitative version of Serre's conjecture on the~CSP. Indeed, Serre's conjecture asserts that an irreducible lattice in a higher-rank semisimple Lie group has the CSP. By Lubotzky-Martin \cite{LubMar2004} it has PRG, and therefore a finite abscissa of convergence. The above conjecture says that this abscissa is in fact independent of the lattice in the group.

{

\begin{remark}[Arithmeticity] Let $H=\prod_{i=1}^m \bfG_i(F_i)$, with each $F_i$ a local field, and each $\bfG_i$ is an absolutely almost simple $F_i$-algebraic group with $\sum_i \mathrm{rank}_{F_i} \bfG_i \ge 2$. If $\Gamma$ is an irreducible lattice in $H$, then 
\begin{enumerate}
    \item All the local fields have the same characteristic;
    \item The $\bfG_i$'s have the same absolute root system $\Phi$;
    \item $\Gamma$ is commensurable to $\bfH(\Gri_S)$, where $\bfH$ is an $\Gri_S$-algebraic group for some ring of $S$-integers in a global field and also has absolute root system $\Phi$. 
\end{enumerate}

 In characteristic zero this result is due to Margulis~\cite[Chapter IX]{Margulis}, and in positive characteristic to Venkataramana~\cite[Theorem 5.39]{Venky1988}. The latter is formulated for groups of adjoint type, however, the theorem holds for arbitrary types. Indeed, it follows from \cite[Prop.~3.19]{MR0316587} that an irreducible lattice in a product of groups of arbitrary type surjects on an irreducible lattice in the corresponding product of the adjoint types with a finite kernel.

\end{remark}

}

        \begin{corABC}\label{cor:LLpositive} Let $H=\prod_i \bfG_i(F_i)$ with $F_i$ local fields with residual characteristic greater than~$3$, and $\bfG_i$ groups of type $A_2$. Let $\Gamma_1$ and $\Gamma_2$ be irreducible lattices in $H$. If~$\alpha_{\Gamma_1}$ and~$\alpha_{\Gamma_2}$ are finite then they are equal. In particular, Serre's conjecture on the CSP for these groups implies the Larsen-Lubotzky conjecture. 
        \end{corABC}

 Corollary~\ref{cor:LLpositive} is a direct consequence of Theorem~\ref{thm:basechange}; see Section~\ref{subsec:proof.B}. The proof of Theorem~\ref{thm:basechange} is based on the following result which is of independent interest.  Let~$\lri$ be a compact discrete valuation ring with maximal ideal $\mfp=(\pi)$ and finite residue field $\kk$. For $\ell \in \N$ let $\lri_\ell=\lri/\mfp^{\ell}$ denote the finite quotient. Let $\Lri \supset\lri$ be a quadratic unramified extension. Let $\GU_d$ and $\SU_d$ be the unitary and special unitary groups, respectively, with respect to $\Lri$; see~Section~\ref{sec:preliminaries}. 

                \begin{thmABC}\label{thm:local.zeta} Let $\lri$ be a compact discrete valuation ring with maximal ideal $\mfp$ and residue field of characteristic $p > 3$. Let $\Lri/\lri$ be a quadratic unramified extension.  The representation zeta functions of $\GL_3(\lri/\mfp^\ell)$, $\GU_3(\lri/\mfp^\ell)$, $\SL_3(\lri/\mfp^\ell)$, $\SU_3(\lri/\mfp^\ell)$, $\SL_3(\lri)$ and $\SU_3(\lri)$ depend only on~$|\lri/\mfp|$. 
        \end{thmABC}
Theorem~\ref{thm:local.zeta} gives an affirmative answer to the $A_2$ case of \cite[Conjecture~1.6]{AKOV2} and to the~$\lambda=\ell^3$ case of \cite[Conjecture~1.3]{Onn2008}. The zeta functions appearing in Theorem~\ref{thm:local.zeta} were explicitly computed in \cite{AKOV2} with restrictions on the characteristic of $\lri$ and $\lri/\mfp$.
Let $e$ denote the absolute ramification index of $\lri$.
In \cite{AKOV2}, these explicit expressions were shown to hold only in residual characteristic $p\geq \min\{3\ell,3e+3\}$ in the case of $\GL_3(\lri/\mfp^\ell)$, $\GU_3(\lri/\mfp^\ell)$, $\SL_3(\lri/\mfp^\ell)$ and~$\SU_3(\lri/\mfp^\ell)$,
and in residual characteristic $p\geq 3e+3$ for $\SL_3(\lri)$ and $\SU_3(\lri)$ when $\lri$ has characteristic~$0$.
Theorem~\ref{thm:local.zeta} implies that those same formulae hold  in residual characteristic $p>3$, independent of the characteristic of $\lri$.
In particular, Theorem~\ref{thm:local.zeta} gives an improvement of the results of \cite{AKOV2} even when $\lri$ has characteristic $0$.
Having explicit formulae for small residue fields such as the field of order $5$ is helpful as these are often the only fields for which brute force experimentation using GAP is feasible.

The rigidity of representation zeta functions reflected in  Theorem \ref{thm:local.zeta} should be contrasted with subgroup growth zeta functions, in which groups over different rings have completely different subgroup growth behavior. Indeed, a pro-$p$ group is $p$-adic analytic if and only if it has polynomial subgroup growth \cite[Ch 4]{LubotzkySegal2003}. Theorem~\ref{thm:local.zeta} represents the main challenge of the present paper as one cannot use the orbit method. We remark that for small primes one does not expect a uniform zeta function; see~\cite{hassain-odd, Singla-Hassain}. 

In the course of proving Theorem~\ref{thm:local.zeta}, we discovered a surprising relation between the representation zeta functions of certain groups which, if generalized, may become a valuable tool to study representations of such groups by dimension reduction; see \S\ref{sec:construction-of-non-regular-characters-G} and \S\ref{sec:proof.E}. To describe it we need more notation. 

Let $\bfG$ be a reductive group defined over $\lri$ and let $\mfg$ denote its Lie algebra. 
The group $\bfG(\lri)$ is profinite, therefore every continuous irreducible representation $\rho$ of $\bfG(\lri)$ factors through a finite congruence quotient. The smallest $\ell=\ell(\rho)$ such that $\rho$ factors through $\bfG(\lri_{\ell+1})$ is called the {\em level} of $\rho$. Fix $\ell \in \N$ and let $\bfG(\lri_{\ell})^{\ell-1}:=\ker\left(\bfG(\lri_\ell) \to \bfG(\lri_{\ell-1})\right)$ be the smallest principal congruence subgroup in $\bfG(\lri_\ell)$. The latter is isomorphic to $(\mfg(\kk),+)$ using the map $I+\pi^{\ell-1} x \mapsto x$, and the conjugation action of $\bfG(\lri_\ell)$ factors through the adjoint action of~$\bfG(\kk)$. Possibly excluding small primes, the fact that $\mfg$ is reductive allows one to use the Killing form to  identify the Pontryagin dual $\mfg(\rf)^\vee$ with $\mfg(\rf)$ as $\bfG(\kk)$-spaces. For every $\Omega \in \bfG(\kk) \backslash \mfg(\kk)$ let $\Irr(\bfG(\lri_\ell) \mid \Omega)$ be the set of equivalence classes of irreducible representation lying above $\Omega$. One can then use Clifford theory to partition the irreducible representations of $\bfG(\lri_\ell)$ 
\begin{equation}\label{eq:irr.orbits}
\Irr(\bfG(\lri_\ell)) = \coprod_{\Omega \in \bfG(\kk) \backslash \mfg(\kk)} \Irr(\bfG(\lri_\ell) \mid \Omega). 
\end{equation}
 To prove Theorem~\ref{thm:local.zeta} we focus on the different families of orbits, and treat each of them by appropriate means. The following two theorems indicate that the analysis of representations lying above each orbit may be reduced to a lower dimensional problem. For decomposable matrices, the correspondence is transparent and given, in arbitrary dimensions, by the following theorem.

        \begin{thmABC}\label{thm:dim.reduction.decomposable} Let $\lri$ be a compact discrete valuation ring with residue field $\kk$ of odd degree and let $\Lri/\lri$ be a quadratic unramified extension with residue field $\KK$. Let $\bfG$ be either $\GL_d$, $\GU_d$, $\SL_d$ or $\SU_d$ for $d \in \N$, and let $\mfg \subset \gl_d$ be the corresponding Lie algebra. For $\SL_d$ and~$\SU_d$ assume further that $p \nmid d$. Let ${\xi} \in \mfg(\kk)$ be a matrix which is conjugate to a block diagonal matrix $\diag({\xi_1},{\xi_2})$ such that ${\xi_i} \in \gl_{d_i}(\KK) \cap \mfg(\kk)$ has characteristic polynomial $f_i(t) \in \KK[t]$ with~$f_1(t)$ and $f_2(t)$ coprime, and let $\bfG_i = \bfG \cap \GL_{d_i}$. Then there exists a bijection
        \[
        \Irr(\bfG(\lri_\ell) \mid [{\xi}]) \xlongleftrightarrow{1:1} \Irr(\bfG_1(\lri_\ell) \mid [{\xi_1}]) \times \Irr(\bfG_2(\lri_\ell) \mid [{\xi_2}])
        \]
with constant dimension ratio 
$[\bfG(\kk):\bfG_1(\kk) \times \bfG_2(\kk)]|\kk|^{\half\left(\dim\bfG-\dim\bfG_1-\dim\bfG_2\right)(\ell-2)}$. 

Consequently
      \[
             \zeta_{\bfG(\lri_\ell), [\xi]}(s)=\frac{|\bfG(\kk)|^{-s}}{|\bfG_1(\kk)|^{-s} | \bfG_2(\kk)|^{-s}}|\kk|^{-\half\left(\dim\bfG-\dim\bfG_1-\dim\bfG_2\right)(\ell-2)s}\zeta_{\bfG_1(\lri_\ell),[\xi_1]}(s)\zeta_{\bfG_2(\lri_\ell),[\xi_2]}(s).
                \]

        \end{thmABC}

        One of the most interesting local observations in this paper is the following \lq dimension reduction\rq~ result, modeled after~Theorem~\ref{thm:dim.reduction.decomposable}. We require the following definition. Let ${\xi \in \mfg(\kk)}$ and let ${\wt{\xi} \in \mfg(\lri_\ell)}$ be a lift. The {\em shadow} of $\wt{\xi}$, denoted $\Sh_{\bfG(\lri_\ell)}(\wt{\xi})$, is the image of its centraliser~$Z_{\bfG(\lri_\ell)}(\wt{\xi})$ in $\bfG(\kk)$; see \cite[Definition 2.2]{AKOV2}. In general, one has $\Sh_{\bfG(\lri_\ell)}(\wt{\xi}) \leq Z_{\bfG(\kk)}({\xi})$.  We say that a lift  $\wt{\xi}$ of $\xi$ is {\em shadow-preserving}, if equality holds.

        \begin{thmABC}\label{thm:dim.reduction.J} Let $\lri$ be a compact discrete valuation ring and let $\Lri/\lri$ be a quadratic ring extension. Let~$\bfG$ be either $\GL_3$ or $\GU_3$ and $\mfg$ be the corresponding Lie algebra. Let $\wt{\xi} \in \mfg(\lri_\ell)$ be a shadow-prerserving lift of $\xi \in \mfg(\kk)$.  Then
        \[
        \zeta_{\bfG(\lri_\ell), [\xi]}(s)=[\bfG(\kk):Z_{\bfG(\kk)}(\xi) ]^{-s}|\kk|^{-\half(\dim \bfG -\dim Z_{\bfG}(\xi)) (\ell-2)s}\zeta_{Z_{\bfG(\lri_{\ell-1})}(\wt{\xi}  \!\!\!\!\mod \pi^{\ell-1})}(s).
        \]
        \end{thmABC}

\subsection{Organisation of the paper} The paper is organised as follows. Section \ref{sec:global} is devoted to global results: we prove Theorem~\ref{thm:basechange} and Corollary~\ref{cor:LLpositive} assuming Theorem~\ref{thm:local.zeta}.  In Section \ref{sec:preliminaries} we set up the notation and collect basic facts about unitary groups over local rings, anti-Hermitian matrices, characters and co-adjoint orbits. Section \ref{sec:local.zeta} is devoted to local zeta functions: we explain the partition of the zeta functions according to orbits, and reduce the proof of Theorem~\ref{thm:local.zeta} to Theorem~\ref{thm:dim.reduction.decomposable}; the construction of the zeta-component associated with the non-regular non-trivial nilpotent orbit (Theorem~\ref{thm:construction}); and the uniformity of the associated zeta function (Theorem~\ref{thm:zeta.e}). Section~\ref{sec:coad.orbits} is devoted to co-adjoint orbits and Section~\ref{sec:proof.thm.D} to a proof of  Theorem~\ref{thm:dim.reduction.decomposable}. The technical core of the paper is in Sections~\ref{sec:construction-of-non-regular-characters-G} and~\ref{sec:proof.E}, where we prove Theorem~\ref{thm:construction} and Theorem~\ref{thm:zeta.e}, respectively. In Section~\ref{sec:proof.E} we compute the representation zeta function of centralisers of certain lifts of the non-regular non-zero nilpotent which leads to the proof of Theorem~\ref{thm:dim.reduction.J}.

\subsection{Acknowledgments}  
The first author was supported by the Australian Research Council, Fellowship FT160100018. The third author acknowledges the support of SERB, India, for financial support through SPG/2022/001099. This collaboration was partially supported by SPARC project SPARC/2018-2019/P88/SL of the Ministry of Human Resource Development, Government of India. {We thank Anup Dixit, Alex Lubotzky, Ehud Meir, Gopal Prasad, K~Srinivas, B~Sury, T~N~Venkataramana, and Christopher Voll, for some very helpful discussions.}

\section{Global results: Proofs of Theorem~\ref{thm:basechange} and Corollary~\ref{cor:LLpositive}}\label{sec:global}

The proof of Theorem~\ref{thm:basechange} follows from a comparison of the Dedekind zeta functions of the function field and the number field cases, general results on Euler products, and Theorem~\ref{thm:local.zeta}.    

\subsection{Generalities on Euler Products}
Let $\Gf$ be a global field.
Let $V(\Gf)$ denote the set of places of $\Gf$.
For each non-Archimedean place $v$, let $q_v$ denote the order of the residue field associated with $v$.
Let $S$ denote a finite set of places of $F$, including the Archimedean ones.
Consider a system of local factors, namely, a family $f = \{f_v(q_v;s)\}_{v\notin S}$ of holomorphic functions defined in the region $\mathrm{Re}(s)>0$ taking positive real values when $s$ is a positive real number. 
We are interested in Euler products of the form
\begin{equation}
  \label{eq:euler-formal}
  \zeta(f;s) = \prod_{v\notin S} (1 + f_v(q_v;s)).
\end{equation}
We say that $\zeta(f;s)$ has abscissa of convergence $s_0$ if the infinite product converges to a holomorphic function on the region $\mathrm{Re}(s)>s_0$.

\begin{definition}
  Let $\alpha$ be a positive real number and $f = \{f_v(q_v;s)\}$ and $g = \{g_v(q_v;s)\}$ be local factors.
  We say that $\zeta(f;s) \equiv_\alpha \zeta(g;s)$ if $\prod_{v\notin S}\frac{1+f_v(q_v;s)}{1+g_v(q_v;s)}$ and $\prod_{v\notin S}\frac{1+g_v(q_v;s)}{1+ f_v(q_v;s)}$ are both holomorphic in the region $\mathrm{Re}(s)>\alpha$.
  As a consequence, $\zeta(f;s)$ and $\zeta(g;s)$ have the same poles and zeroes (taken with multiplicity) in the region $\mathrm{Re}(s)>\alpha$.
\end{definition}
\begin{lemma}
  \label{lemma:f-g}
  Suppose that $f = \{f_v(q_v;s)\}$ and $g=\{g_v(q_v;s)\}$ are local factors such that $\prod_{v\notin S}(1 + f_v(q_v;s) - g_v(q_v;s))$ has abscissa of convergence $\alpha$.
  Then $\zeta(f;s)\equiv_\alpha \zeta(g;s)$.
\end{lemma}
\begin{proof}
  We have:
  \begin{displaymath}
    \frac{1+f_v}{1+g_v} = 1 + \frac{f_v-g_v}{1+g_v} < 1 + (f_v-g_v).
  \end{displaymath}
  Interchanging the roles of $f_v$ and $g_v$, a similar bound can be found for $\frac{1+g_v}{1+f_v}$, whence the lemma follows.
\end{proof}
\begin{corollary}
  \label{corollary:hgtoq}
  If $\prod_v(1+f_v(q_v,s))$ has abscissa of convergence $\alpha$, and $\prod_v (1+g_v(q_v;s))$ has abscissa of convergence $\beta$, and $\alpha<\beta$, then
  \begin{displaymath}
  \prod_v (1 + f_v(q_v;s) + g_v(q_v;s))\equiv_\alpha \prod_v (1 + g_v(q_v;s)).
\end{displaymath}
\end{corollary}
\begin{corollary}
  \label{corollary:sum_to_prod}
  Suppose that $\prod_v (1+f_v(q_v,s))$ and $\prod_v (1+g_v(q_v;s))$ have abscissa of convergence $\beta$, and $\prod_v(1 + f_v(q_v;s)g_v(q_v;s))$ has abscissa of convergence $\alpha<\beta$, then
  \begin{displaymath}
    \prod_v(1 + f_v(q_v;s) + g_v(q_v;s))\equiv_\alpha \prod_v (1 + f_v(q_v;s))\prod_v(1+g_v(q_v;s)).
  \end{displaymath}
\end{corollary}
\begin{lemma}
  \label{lemma:fg-approx}
  Suppose $f$ and $g$ are monic polynomials of degree $a$ and $b$ respectively, then for every $s>0$, there exists $\kappa>0$ such that
  \begin{displaymath}
    |f(q)g(q)^{-s}-q^{a-bs}| < q^{(a-1)-bs}\kappa
  \end{displaymath}
  for sufficiently large $q$.
\end{lemma}
\begin{proof}
  Write $f(q) = q^a(1+q^{-1}f_0(q^{-1}))$ and $g(q)=q^b(1+q^{-1}g_0(q^{-1}))$ for some polynomials $f_0$ and $g_0$.
  Then, for each fixed $s$,
  \begin{displaymath}
    \frac{f(q)}{g(q)^s}-q^{a-bs} = q^{a-bs}\left(\frac{1+q^{-1}f_0(q^{-1})}{(1+q^{-1}g_0(q^{-1}))^s}-1\right) = q^{a-bs}(q^{-1}f_0(q^{-1})-sq^{-1}g_0(q^{-1})+o(q^{-1})),
  \end{displaymath}
  from which the lemma follows.
\end{proof}
\begin{lemma}
  Suppose $f$ and $g$ are monic real polynomials of degree $a$ and $b$ respectively, with $b>0$.
  Then $\prod_v(1+f(q_v)g(q_v)^{-s})$ has abscissa of convergence $(a+1)/b$, can be extended to a meromorphic function slightly to the left of this abscissa, and has a simple pole at $s=(a+1)/b$.
\end{lemma}
\begin{proof}
  The Dedekind zeta function of a global field $\Gf$ is given by
  \begin{displaymath}
    \zeta_{\Gf}(s) = \prod_v(1-q_v^{-s})^{-1},
  \end{displaymath}
  where the product is over all finite places of $\Gf$.
  It is well-known that $\zeta_{\Gf}(s)$ has abscissa of convergence $s=1$.
  Moreover, it can be extended to a meromorphic function on the complex plane whose only pole, which is simple, is at $s=1$; see  \cite{lang2013algebraic,MR1876657}.
  It follows that $\prod_v(1+q_v^{-s})$ also has abscissa of convergence $1$ and can be extended to a meromorphic function slightly to the left of this abscissa.
  By a change of variables, $\prod_v(1+q^{a-bs})$ has abscissa of convergence $(a+1)/b$, and can be extended to a meromorphic function slightly to the left of $(a+1)/b$.
  By Lemmas~\ref{lemma:f-g} and~\ref{lemma:fg-approx},
  \begin{displaymath}
      \prod_v (1+f(q_v)g(q_v)^{-s}) \equiv_{\frac ab} \prod_v(1+q^{a-bs}),
  \end{displaymath}
  hence it also has a simple pole at $s=(a+1)/b$ and can be extended to an analytic function slightly to the left of $\mathrm{Re}(s)=(a+1)/b$.
\end{proof}
\subsection{Proof of Theorem~\ref{thm:basechange}}
Up to isogeny, every absolutely simple group $\bfG$ of type $A_2$ is one of the following (see \cite[Table~II]{MR0224710}): 
\begin{enumerate}
\item $\SL_3(\Gf)$.
\item $\SU_3(\Gf,f)$ defined with respect to a quadratic extension $L/\Gf$ and a nondegenerate Hermitian form $f$.
\item $\SL_1(D)$, where $D$ is a division algebra of degree $3$ over $\Gf$.
\item $\SU_1(D,f)$, with respect to a quadratic extension $L/\Gf$, where $D$ is a division algebra of degree~$3$ over $L$ and $f$ is a nondegenerate Hermitian form.
\end{enumerate}

Using the Euler product decomposition~\eqref{eq:euler}, the possible local factors are generically the representation zeta functions of $\bfG(\Gri_v)$ for either $\SU_3(\Gri_v)$ or $\SL_3(\Gri_v)$. The form (3) can occur only in finitely many places and does not affect the analytic properties of the product. Indeed, by \cite[Theorem I]{Shechter} the representation zeta function in this case is rational with a simple pole at $s=2/3$. The form (4) does not occur when $F$ is a local field; see \cite[Table~II]{MR0224710}.

For types (1) and (2) in the number field case it was shown in \cite[Corollary~D]{AKOV2} that the local zeta function $\zeta_{\bfG(\Gri_v)}(s)$ in both these cases is of the form
\begin{displaymath}
  \zeta_{\bfG(\Gri_v)}(s) = 1 + h_0(q_v)g_0(q_v)^{-s} + \tfrac 12 h_1(q_v)g_1(q_v)^{-s} + \tfrac 13 h_2(q_v)g_2(q_v)^{-s} + \tfrac 16 h_3(q_v)g_3(q_v)^{-s} + E_v(q_v;s),
\end{displaymath}
where $h_0,\dotsc,h_3$ and $g_0,\dotsc,g_3$ are monic polynomials with $\deg h_0=1$, $\deg g_0=2$, $\deg h_1 = \deg h_2 = \deg h_3 = 2$, $\deg g_1 = \deg g_2 = \deg g_3 = 3$, and $\prod_v(1 + E_v(q_v;s))$ has abscissa of convergence at most~${\frac 56}$.
When $\Gf$ is a function field, Theorem~\ref{thm:local.zeta} tells us that the same local factors appear in the Euler product for $\zeta_{\bfG(\Gri_S)}(s)$.
The polynomials $h_i$ and $g_i$ depend only on whether $\bfG(\Gri_v)$ is $\SU_3(\Gri_v)$ or $\SL_3(\Gri_v)$.

Applying Lemma~\ref{lemma:fg-approx} (note that since there are only two variants of local zeta factors, the constant~$\kappa$ in Lemma~\ref{lemma:fg-approx} can be chosen to work for all but finitely many places), we see that 
\begin{displaymath}
  \prod_v\zeta_{\bfG(\Gri_v)}(s) \equiv_{\frac 23} \prod_v\left(1+ q_v^{1-2s} + q_v^{2-3s} + E_v(q_v;s)\right),
\end{displaymath}
where the abscissa of convergence of $E_v(q_v;s)$ is at most $5/6$.
Applying Corollaries~\ref{corollary:hgtoq} and~\ref{corollary:sum_to_prod}, Theorem~\ref{thm:basechange} follows.

\subsection{Proof of Corollary~\ref{cor:LLpositive}}\label{subsec:proof.B} Let $I$ be a finite set and let $\Gamma$ an irreducible lattice in $H=\prod_{i \in I} \bfG_i(F_i)$, where each $F_i$ is a local field of characteristic different from $2$ and $3$, each is $\bfG_i$ a connected, almost simple $F_i$-group, and $\sum_{i \in I}\mathrm{rank}_{F_i}(\bfG_i) \ge 2$.  By arithmeticity \cite{Margulis, Venky1988}, there exists a finite set of places~$S$ of $\Gf$ and a continuous map 
\[
\prod_{v \in S} \bfG(\Gri_v) \longrightarrow H
\]
such that $\bfG$ is of type $A_2$ and the image of $\bfG(\Gri_S)$ is commensurable to $\Gamma$. Since the abscissa of convergence is a commensurability invariant, if $\Gamma$ has polynomial representation growth, then so does~$\bfG(\Gri_S)$. We can then use the Euler product decomposition \eqref{eq:euler} and deduce that ${\alpha_\Gamma=\alpha_{\bfG(\Gri_S)}=1}$.

\section{Preliminaries and notation}
\label{sec:preliminaries}

Let $\lri$ be a complete discrete valuation ring with residue field $\rf$ of cardinality $q$ and odd characteristic~$p$.  Let~$\mfp$ be the maximal ideal and let $\pi$ be a fixed uniformiser. Let $\Lri$ be an unramified quadratic extension. It follows that there exits $\nonsq \in \Lri$ with $\nonsq^2 \in \lri^\times \smallsetminus (\lri^\times)^2$ such that $\Lri=\lri[\nonsq]$. Let $\mfP=\pi \Lri$ be the maximal ideal in $\Lri$ and $\Rf=\Lri/\mfP$ the residue field, a quadratic extension of $\rf$ generated by the image of $\nonsq$. For $\ell \in \N$, we let $\lri_\ell=\lri/\mfp^\ell$ and  $\Lri_\ell=\Lri/\mfP^\ell$ denote the finite quotients. We denote by $x \mapsto x^\circ$ the non-trivial Galois automorphism of $\Lri/\lri$, characterised by $\nonsq^\circ= -\nonsq$.

\subsection{The unitary group and its Lie algebra} The form of the unitary group which is most suitable for our applications is defined as follows. Let $W \in \GL_d$ denote the permutation matrix corresponding to the longest Weyl element, that is $W=(w_{i,j})$ such that $w_{i,j} = 1$ for $j =   d-i+1$ and $0$ otherwise. Consider the involution on $\Mat_d(\Lri)$ defined by
\begin{equation}\label{staroperation}
  (a_{i,j})^\star=W(a_{j,i}^\circ)W^{-1},
\end{equation}
and its associated Hermitian form on $\Lri_\ell^d$
\[
  \langle u,v \rangle_\star=\sum_{i=1}^d v_{i}^\circ u_{d+1-i}.
\]

For $\ell \in \N \cup \{\infty\}$ the unitary group with respect to $\star$ and its Lie algebra of anti-Hermitian matrices are given by
\[
  \begin{split}
    \GU_d(\lri_\ell) &= \left\{ A \in \GL_d(\Lri_\ell) \mid A^\star A=\mathrm{I}_d \right\}, \\
    \gu_d(\lri_\ell) &= \left\{ A \in \gl_d(\Lri_\ell) \mid A+ A^\star =0_d\right\}.
  \end{split}
\]
Observe that $A=(a_{i,j}) \in \gu_d(\lri_\ell)$ if and only if $a_{i,j}+a_{d+1-j,d+1-i}^\circ=0$ for $1 \leq i,j \leq d$. We write $\SU_d=\SL_d \cap \GU_d$ for the special unitary group and $\su_d$ for its Lie algebra. 

\smallskip

Throughout this paper we consider  $\GL_d$, $\SL_d$, $\GU_d$ and $\SU_d$ as $\lri$-group schemes, where the $R$-points of the latter are the fixed points of $A \mapsto (A^\star)^{-1}$ for every $\lri$-algebra $R$ and $A \in \Mat_d(R)$. We also consider $\gl_d$, $\sl_d$, $\gu_d$ and $\su_d$ as $\lri$-Lie algebra schemes, the latter being the fixed points of $A \mapsto -A^\star$. The adjoint action of a group on its Lie algebra will be denoted by $\Ad$.

\smallskip
We collect a few facts about the Hermitian form, the unitary group, and the Lie algebra of anti-Hermitian matrices which will be used throughout.

\begin{enumerate}

  \item {\bf Uniqueness of Hermitian structure.} Up to conjugation there is a unique Hermitian structure on $\Lri_\ell^d$ with respect to~$\circ$.

  \item {\bf Rigidity of conjugation.} If $A,B \in \gu_d(\lri_\ell)$ are $\Ad(\GL_d(\Lri_\ell))$-conjugate then they are $\Ad(\GU_d(\lri_\ell))$-conjugate; \cite[Proposition 3.2]{AKOV2}.

  \item {\bf Characterisation of anti-Hermitian similarity classes.} Let $A \in \Mat_d(\Lri_\ell)$ be $\GL_d(\Lri_\ell)$-conjugate to an anti-Hermitian matrix with respect to $\star$. Let $f_A(t)=t^d+\sum_{i=0}^{d-1} c_it^i  \in \Lri_\ell[t]$ be its characteristic polynomial. Then
        \begin{equation}\label{cond.Hermitian}
          c_i=(-1)^{d-i}c_{i}^\circ, \quad 0 \le i < d.
        \end{equation}
        The converse holds for cyclic matrices; see \cite[Lemma 3.5]{AKOV2}. 
\end{enumerate}

\subsection{Characters and co-adjoint orbits}\label{subsec:char.and.coad.orbits} In what follows we shall make extensive use of the Pontryagin dual ${\mfg(\lri_\ell)}^\vee=\Hom_{\mathrm{Groups}}(\mfg(\lri_\ell), \C^\times)$, where $\bfG \in \{ \GL_d, \GU_d, \SL_d, \SU_d\}$ and $\mfg = \mathrm{Lie}(\bfG)$, and of the co-adjoint orbits $\Ad(\bfG(\lri_\ell)) \backslash \mfg(\lri_\ell)^\vee$.
We first give a convenient unified description of the Pontryagin duals; see \cite[Lemma 5.12]{AKOV2} and \cite[Remark 5.13]{AKOV2} for more details. We fix a primitive additive character $\psi:\Lri_\ell \to \C^\times$, namely $\psi_{|\mfP^{\ell-1}} \ne 1$, such that post-composition with $\psi$ identifies
\[
  \Hom_{\Lri_\ell\text{-Mod}}\left( \Lri_\ell^N,\Lri_\ell \right) \cong \Hom_{\mathrm{Groups}}( \Lri_\ell^N, \C^\times).
\]

In the present setup, we can use the trace form to obtain an equivariant identification of $\mfg(\lri_\ell)^\vee$ and $\mfg(\lri_\ell)$ for $\ell \in \N$. Indeed, excluding primes dividing $d$ for $\SL_d$ and $\SU_d$, we have a non-degenerate pairing
\[
  \gl_d(\Lri_\ell) \times \gl_d(\Lri_\ell) \longrightarrow \Lri_\ell, \quad (A,B) \mapsto \tr(AB),
\]
which descends to a non-degenerate pairing $\mfg(\lri_\ell) \times \mfg(\lri_\ell) \rightarrow \lri_\ell$ because in both cases~$\mfg(\lri_\ell)$ is a rational form of~$\gl_d(\Lri_\ell)$.  This gives rise to an equivariant identification of $\mfg(\lri_\ell)$ with its (Pontryagin and linear) dual. Under this identification, we denote the character corresponding to $A \in \mfg(\lri_\ell)$ by $\psi_A \in \mfg(\lri_\ell)^\vee$. With this identification, one can describe the dual of a subgroup of $\mfg(\lri_\ell)$ in a compatible manner. For an additive subgroup (sub Lie algebra in our applications) $\mfh < \mfg(\lri_\ell)$ let
\[
  \mfh^\perp=\left\{ A \in \mfg(\lri_\ell) \mid \tr(A,B)=0,  \text{for all $B \in \mfh$}\right\}.
\]
Then, using the non-degenerate paring, we have  $\mfh^\vee \cong \mfg(\lri_\ell)/\mfh^\perp$.

\smallskip

Using the equivariant identification of the Lie algebra and its Pontrayagin dual the co-adjoint orbits can be described in terms of the adjoint orbits. It is convenient to use the tower of groups and Lie algebras obtained by the reduction maps $\Mat_d(\Lri_{\ell+1}) \to \Mat_d(\Lri_{\ell})$, for $\ell \in \N$, to study the adjoint orbits inductively as we have canonical surjections $\bfG(\lri_{\ell+1}) \backslash \mfg(\lri_{\ell+1}) \to \bfG(\lri_\ell) \backslash \mfg(\lri_\ell)$.

\section{Local zeta functions and proof of Theorem~\ref{thm:local.zeta}}\label{sec:local.zeta}

\subsection{Decomposition of the local zeta function}

Let $\bfG$ be one of the groups $\GL_d$, $\GU_d$, $\SL_d$ or $\SU_d$. Recalling the partition~\eqref{eq:irr.orbits} 
we can  decompose the representation zeta function accordingly 
\begin{equation}\label{eq:zeta.over.orbit}
  \begin{split}
    \zeta_{\bfG(\lri_\ell)}(s)&=\sum_{\Omega \in \bfG(\rf) \backslash \mfg(\kk)} \zeta_{\Omega}(s),\quad \text{where} \\
    \zeta_{\Omega}(s) &=\zeta_{\bfG(\lri_\ell),\Omega}(s)= \sum_{\chi \in \Irr(\bfG(\lri_\ell) \mid \Omega)} \chi(1)^{-s}.
  \end{split}
\end{equation}

\smallskip

An orbit is called regular, irreducible, scalar, nilpotent, or non-primary decomposable, if the matrices in the orbit have that property. By non-primary decomposable we refer to matrices with characteristic polynomial $f(x)=f_1(x)f_2(x)$ with $f_1$ and $f_2$ co-prime. The difficulty of constructing the irreducible representations of $\bfG(\lri_\ell)$ varies significantly with the type of orbit. The construction of irreducible representations lying over regular orbits for the general linear groups was initiated by Hill \cite{Hill_Reg} and then completed in \cite{KOS} for odd residual characteristic and in \cite{StasinskiStevens} for arbitrary characteristic; the construction in \cite{KOS} includes also the unitary groups. Representations lying over scalar matrices are one-dimensional twists of lower-level representations. Irreducible representations lying over irreducible orbits give rise to supercuspidal representations of the group over the corresponding local field; they were studied in \cite{AOPS}. The non-primary decomposable representations are dealt with in Section~\ref{subsec:decomposable.orbits.reps} below. The most challenging irreducible representations are those lying over non-regular nilpotent orbits.

\subsection{Local zeta functions in type $A_1$ and $A_2$} Explicit formulae for the representation zeta function of groups of type $A_1$ and $A_2$ in characteristic zero were obtained in \cite{AKOV2}. Rather than repeating the quite involved formulae in {\em loc. cit.} we break the problem into lower dimensional building blocks and explain why the formulae are independent of the characteristic of the local field and depend only on the cardinality of the residue field.

\subsubsection{Scalar orbits} Irreducible representations lying over scalar orbits are twists by one-dimensional characters of lower-level representations. Indeed, the one-dimensional representations of $\bfG(\lri_\ell)$, for $\bfG \in \{\GL_d,\GU_d\}$ are precisely the pull-backs from $\bfG_1(\lri_\ell)$ where $\bfG_1$ is $\GL_1$ or $\GU_1$ along the determinant map. Indeed, if the restriction of $\rho$ to $\bfG(\lri_\ell)^{\ell-1}$ is represented by a scalar matrix $\alpha \Id_d$, let $\chi:\bfG_1(\lri_\ell) \to \C^\times$ be a character extending  
\[
\rho(I+\pi^{\ell-1}X)=\psi(\tr(\pi^{\ell-1}\alpha X)).
\]
If $\chi: \bfG_1(\lri_\ell) \to \C^\times$ is a one-dimensional representation extending $\rho$, then $\rho \otimes \chi^{-1}$ lies above the zero matrix, and therefore factors through $\bfG(\lri_{\ell-1})$. The special linear and special unitary groups are perfect and have no non-trivial one-dimensional representations.   

\subsubsection{Representations over regular orbits and their zeta function in type $A_d$} For regular orbits $\Omega \in \bfG(\rf) \backslash \mfg(\rf)$ the  corresponding representation zeta function of $\bfG(\lri_\ell)$ was computed in \cite{KOS} for all groups of type $A_d$ with odd residual characteristic. They are given, uniformly and independently of the characteristic of $\lri$, by
\begin{equation}\label{eq:zeta.regular}
\zeta_{\bfG(\lri_\ell)}^{\text{reg}}(s) =\sum_{x \in \bfG(\kk)\backslash\mfg(\kk)^{\mathrm{reg}}}q^{(\ell-2)d} \left|\mathrm{C}_{\bfG(\kk)}(x)\right|\left(q^{\left(d \atop 2\right)(\ell-2)}\frac{|\bfG(\kk)|}{|\mathrm{C}_{\bfG(\kk)}(x)|}\right)^{-s}.
\end{equation}
See \cite{KOS} for more details and combinatorial formulae.

For groups of type $A_1$, all characters are either scalar or regular. Therefore, the zeta function $\zeta_{\bfG(\lri_\ell)}(s)$, for $\bfG \in  \{\GL_2, \GU_2, \SL_2, \SU_2\}$ is uniform and can be recursively computed by 
\[
  \begin{split}
    \zeta_{\bfG(\lri_\ell)}(s)&=\zeta_{\bfG(\lri_\ell)}^{\text{reg}}(s)+q^{\dim Z(\mfg(\kk))}\zeta_{\bfG(\lri_{\ell-1})}(s).
  \end{split}
\]

For groups of type $A_2$, we need to account for non-regular non-primary decomposable orbits and the primary ones. 

\subsubsection{Non-primary orbits in type $A_d$}\label{subsec:decomposable.orbits.reps} The zeta function of associated non-primary orbits can be reduced to dilatations of products of the zeta function of the corresponding parts as given in Theorem~\ref{thm:dim.reduction.decomposable}, for arbitrary~$d$. Explicating the formulae for $d=3$ we obtain the following.

\begin{corollary}\label{cor:zeta.over.abb} Let $\bfG \in \{\GL_3, \GU_3, \SL_3, \SU_3\}$ and let $\xi=\diag(a,b,a) \in \mfg(\rf)$ with $a \ne b$, $a+a^\circ=b+b^\circ=0$. Then
    \[
      \zeta_{\bfG(\lri_\ell) , [\xi]}(s)= \frac{|\bfG(\lri_{\ell_1})|^{-s}}{|\bfG_1(\lri_{\ell_1})|^{-s} | \bfG_2(\lri_{\ell_1})|^{-s}} \zeta_{\bfG_1(\lri_{\ell-1})}(s) \zeta_{\bfG_2(\lri_{\ell-1})}(s),
    \]
with 
\begin{itemize}
    \item $\bfG_1=\GL_1$, $\bfG_2=\GL_2$ if $\bfG=\GL_3$ 

    \item $\bfG_1=\GU_1$, $\bfG_2=\GU_2$ if $\bfG=\GU_3$ 

    \item $\bfG_1=\{1\}$, $\bfG_2=\GL_2$ if $\bfG=\SL_3$ 

    \item $\bfG_1=\{1\}$, $\bfG_2=\GU_2$ if $\bfG=\SU_3$ 

\end{itemize}

\end{corollary}

\begin{remark} Note that the matrix $\diag(a,a,b)$, which conforms to the setup in Theorem~\ref{thm:dim.reduction.decomposable}, is anti-Hermitian with respect to the form~$\circ$, and not with respect to the form $\star$. The fact that up to conjugation there is a unique anti-Hermitian form allows us to go back and forth between the different forms. The form $\circ$ turns out to be more useful for decomposable non-primary orbits; see the proof of Proposition~\ref{prop:decomp.orbits}.

\end{remark}

\subsubsection{Primary non-regular non-central orbits in type $A_2$} For $\mfg$ of type $A_2$ up to a translation by a scalar matrix in $\mfg(\rf)$ one reduces to the unique non-regular nilpotent orbit, which we represent by the matrix $E \in \gl_3$ and  $\nonsq E \in \gu_3$,  where
\begin{equation}\label{def.E}
  E=E_{1,3}=\left[\begin{matrix} 0 & 0 & 1 \\ 0 & 0 & 0 \\0 & 0 & 0 \end{matrix}\right].
\end{equation}

To complete the construction of all irreducible representations of $\bfG(\cO_\ell)$, it remains to describe all non-regular nilpotent non-central irreducible representations of $\bfG(\lri_\ell)$. Indeed, the representations above all other primary non-regular non-central orbits are obtained by one-dimensional twist of the nilpotent ones. We dedicate Section~\ref{sec:construction-of-non-regular-characters-G} to give their construction (see Theorem~\ref{thm:construction}) which is independent of the ring. 
The  discussion above gives a blueprint proof of Theorem~\ref{thm:local.zeta} pending proofs of Theorems~\ref{thm:dim.reduction.decomposable} and Theorem~\ref{thm:construction}.

\subsection{Proof of Theorem~~\ref{thm:local.zeta} (modulo Theorem~\ref{thm:dim.reduction.decomposable} and Theorem~\ref{thm:zeta.e}).} 
The representation zeta function can be written as a sum over orbits as in~\eqref{eq:zeta.over.orbit}. The regular zeta function \eqref{eq:zeta.regular} is uniform by \cite{KOS}. The zeta function over decomposable orbits of type $\diag(a,b,b)$, with $a \ne b$, is uniform by Corollary~\ref{cor:zeta.over.abb}. The zeta function over the remaining non-regular nilpotent (up to scalar) orbit is uniform by Theorem~\ref{thm:zeta.e} and Remark~\ref{rem:SL3.SU3}.  

\section{Co-adjoint orbits}\label{sec:coad.orbits}

The main difficulty in constructing and classifying the irreducible representations of the groups $\bfG(\lri_\ell)$ in the positive characteristic case is the absence of the Orbit method. The strategy we follow is to classify orbits of irreducible representations of the largest abelian principal congruence subgroup and construct representations over each orbit separately. Since the largest principal congruence subgroup of $\bfG(\lri_\ell)$ is $\exp\left(\pi^{\lup}\mfg(\lri_{\ldown})\right) \cong \left(\mfg(\lri_{\ldown}),+\right)$, with $\ldown=\lfloor \ell/2 \rfloor$ and $\lup=\lceil \ell/2 \rceil$, we are led to classify (co)adjoint orbits.

\subsection{Non-primary decomposable orbits}\label{subsec:decomposable.orbits} We now explain how non-primary decomposable orbits can be approached using induction on the dimension. We start with $\Ad(\bfG(\rf))$-orbits on~$\mfg(\rf)$.  Suppose that $\xi \in \mfg(\rf)$ has characteristic polynomial $f \in \Rf[t]$, where $f \in \rf[t]$ if $\mfg=\gl_d$ or satisfies \eqref{cond.Hermitian} if $\mfg=\gu_d$, and that $f(t)=f_1(t)f_2(t)$ with $f_1$ and $f_2$ co-prime {such that~$\xi$ is} $\Ad(\bfG(\kk))$-conjugate to $j(\xi_1,\xi_2) \in \mfg(\rf)$
with {$\xi_i$ in a $\kk$-Lie subalgebra $\mfg_{i}(\rf)$ and $j: \mfg_1(\rf) \oplus \mfg_2(\rf) \hookrightarrow \mfg(\rf)$. We assume that we have a corresponding embedding $\bfG_1 \times \bfG_2 \hookrightarrow \bfG$ and that all these embeddings can be lifted to $\mfg(\lri_\ell)$ and $\bfG(\lri_\ell)$.}

\begin{proposition}\label{prop:decomp.orbits} With the above notation, the following hold.
  \begin{enumerate}
    \item Every $\wt{\xi} \in \mfg(\lri_\ell)$ that lies over $\xi$ can be conjugated to $j(\wt{\xi_1}, \wt{\xi_2})$ for appropriate $\wt{\xi_i} \in \mfg_i(\lri_\ell)$.

   \item Two lifts $\wt{\xi}=j(\wt{\xi_1},\wt{\xi_2})$ and $\wt{\eta}=j(\wt{\eta_1},\wt{\eta_2})$ in $\mfg(\lri_\ell)$ of $\xi=j(\xi_1,\xi_2)$ are $\bfG(\lri_\ell)$-conjugate if and only if $\wt{\xi_i},\wt{\eta_i} \in \mfg_i(\lri_\ell)$ are $\bfG_i(\lri_\ell)$-conjugate. 
   
   In particular,  
   \[
   C_{\bfG(\lri_\ell)}(\wt{\xi}) = C_{\bfG(\lri_\ell)}(\wt{\xi}) \cap  j\left(\bfG_1(\lri_\ell)\times \bfG_2(\lri_\ell)\right) \cong C_{\bfG_1(\lri_\ell)}(\wt{\xi_1}) \times C_{\bfG_2(\lri_\ell)}(\wt{\xi_2}).
   \]

  \end{enumerate}
  It follows that there is a bijection between $\Ad(\bfG(\lri_\ell))$-orbits in $\mfg(\lri_\ell)$ that lie over~$\xi$ and the product of pairs of orbits from $\Ad(\bfG_i(\lri_\ell))$-orbits in $\mfg_i(\lri_\ell)$ lying above $\xi_i$.

\end{proposition}

\begin{proof} To prove this proposition it is convenient to use the Hermitian form ${(a_{i,j}) \mapsto (a_{j,i}^\circ)}$, since it preserves block diagonal matrices.  

  \begin{enumerate}

    \item Let
          \[
            g = \left[\begin{matrix}  I_1 & \pi V\\ 0  & I_2 \end{matrix}\right] \in \GL_d(\Lri_\ell) \quad \text{and} \quad \wt{\xi}=\left[\begin{matrix}  A_1 & \pi B \\ \pi C  &  A_2 \end{matrix}\right] \in \mfg(\lri_\ell),
          \]
          with $A_i$ lying over $\xi_i$, and $I_i$ is the appropriate identity matrix of dimension $d_i=\dim \mfg_i$. 
          Then
          \begin{equation}\label{eq:conjugate}
            g\wt{\xi} g^{-1}=\left[\begin{matrix}  A_1+\pi^2 V C& \pi (B+VA_2-A_1V-\pi VCV) \\ \pi C  &  A_2-\pi^2CV \end{matrix}\right].
          \end{equation}

          We claim first that there exists $\pi V \in \Mat_{d_1 \times d_2}(\pi \Lri_\ell)$ such that the top right block in~\eqref{eq:conjugate} is zero. To prove this, observe that the map from $\Mat_{d_1 \times d_2}(\Lri_{\ell-1})$ to itself given by
          \[
            V \mapsto VA_2-A_1V-\pi VCV {\pmod  {\mfP^{\ell-1}}}
          \]
          is injective, and therefore surjective. For $\ell=2$ that follows from the fact that the characteristic polynomials $f_1$ and $f_2$ of the images of $A_1$ and $A_2$, respectively, modulo the prime ideal are coprime. Indeed, $VA_1=A_2V$ over~$\Rf$ implies that~${Vf(A_1)=f(A_2)V}$ for every polynomial $f$, and taking $f=f_1$ gives $0=Vf_1(A_1)=f_1(A_2)V$ with~$f_1(A_2)$ being invertible, hence $V=0$.  Hensel's lemma then implies that the map is surjective for every $\ell$. In the general linear case the claim follows by replacing  $\Lri_\ell$ with~$\lri_\ell$ throughout. For the unitary case we use the fact that every two anti-Hermitian matrices in~$\mfg(\lri_\ell)$ which are $\GL_d(\Lri_\ell)$-conjugate are already $\GU_d(\lri_\ell)$-conjugate; see Proposition~3.2 in \cite{AKOV2}.

          As the lower left blocks of $\wt{\xi}$ and $g \wt{\xi} g^{-1}$ remained the same we can repeat the same process using a transposed version to eliminate $\pi C$ as well.  The resulting element is now in $j\left(\mfg_1(\lri_\ell) \oplus \mfg_2(\lri_\ell)\right) \subset \mfg(\lri_\ell)$ and lies above $j(\xi_1,\xi_2)$.

    \item  Let $A=j(A_1,A_2), B=j(B_1,B_2) \in  \mfg(\lri_\ell)$ and let
          \[
            g= \left[\begin{matrix}  V_{11} & V_{12} \\ V_{21}  & V_{22} \end{matrix}\right] \in \bfG(\lri_\ell).
          \]
          Explicating $g A= B g$ we get that $V_{ij}A_j=B_{i} V_{ij}$ for $i,j \in \{1,2\}$.
          A similar argument to the one given in part (1) now implies that $V_{12}=0$ and $V_{21}=0$. It then follows that~$V_{ii}$ is invertible and conjugates $A_i$ to $B_i$. \qedhere
  \end{enumerate}

\end{proof}

\subsection{Primary orbits}\label{sec:adjoint-orbits-above-E} The previous section  reduced the classification of adjoint orbits to the primary orbits. For $\GL_d$ the results \cite[\S2]{Hill_Jord}  show that one can reduce the classification of orbits in~$\gl_d(\lri_\ell)$ to orbits lying over nilpotent orbits in~$\gl_d(\rf)$. These results directly apply to $\GU_d$ thanks to Proposition~3.2 in \cite{AKOV2} that guarantees that anti-Hermitian matrices in $\gu_d(\lri_\ell)$ are $\GL_d(\Lri_\ell)$-conjugate if and only if they are $\GU_d(\lri_\ell)$-conjugate.

The classification problem of elements in $\mfg(\lri_\ell)$ lying over non-regular nilpotent elements in~$\mfg(\rf)$ is notoriously difficult. The case of groups of type $A_2$ consists of one such orbit and is completely solved in \cite[\S3]{AKOV2}. We now describe and rephrase the latter classification in terms more suitable for the present paper.

Let $\phi_{\ell,j}: \gl_d(\Lri_\ell) \to \gl_d(\Lri_j)$ be the reduction map for $1 \leq j \leq \ell$. The map $\phi_{\ell,j}$ descends to a map between adjoint orbits
\[
  \bfG(\lri_\ell) \backslash \mfg(\lri_\ell) \longrightarrow  \bfG(\lri_j) \backslash \mfg(\lri_j).
\]

Proposition~\ref{fibre.over.e} gives a (redundant) set of representatives for the $\Ad(\bfG(\lri_\ell))$-orbits in the fibre $\phi^{-1}_{\ell,1}(\Ad(\GL_3(\Rf))E \cap \mfg(\rf))$. We note that
$\Ad(\GL_3(\Rf))E \cap \mfg(\rf)$ equals $\Ad(\bfG(\rf))E$ if $\bfG=\GL_3$ and equals $\Ad(\bfG(\rf))\nonsq E$ if $\bfG=\GU_3$.

\begin{lemma} For any local ring $R$ of odd characteristic such that $E=E_{1,3} \in \gl_3(R)$, we have
  \[
    \begin{split}
      \mfh(R):=C_{\gl_3(R)}(E)&=\left\{\left[\begin{matrix} a & x & z \\  0 & b & y \\ 0&  0 & a \end{matrix}\right] \mid a,b,x,y,z \in R\right\},  \\
      \mfh(R)^\perp=[E,\gl_3(R)]&=\left\{\left[\begin{matrix} a & x & z \\  0 & 0 & y  \\ 0&  0  & -a \end{matrix}\right] \mid a,x,y,z \in R \right\},
    \end{split}
  \]
  and $\gl_3(R)=\mfh(R)^\Tr \oplus \mfh(R)^\perp$.
\end{lemma}

\begin{proof}
This follows by usual matrix multiplication. 
\end{proof}

\begin{proposition}\cite[Sections~2,3]{AKOV2} \label{fibre.over.e}  Every $\Ad(\bfG(\lri_\ell))$-orbit $\Omega \subset \mfg(\lri_\ell)$ that lies above \linebreak $\Ad(\GL_3(\Rf))E \cap \mfg(\rf)$ has a representative $A \in \mfh(\Lri_\ell)^\Tr \cap \mfg(\lri_\ell)$ of the form
  \begin{equation}\label{form.of.char} 
    A = A(\ell, m; \omega, \delta,  \alpha, \beta, \gamma)  
        =  r E +  \omega \mathrm{I} + \underbrace{\left[\begin{matrix}0  & 0 & 0 \\   0 & \delta & 0 \\ \delta^2 & 0  & 0 \end{matrix}\right]}_{\Delta} +  \left[\begin{matrix} 0 & 0 & 0 \\   \alpha & 0 & 0 \\ \gamma &  \beta & 0 \end{matrix}\right]
  \end{equation}
  with uniquely determined
  \[
    m = \min \{\val(\alpha), \val(\beta), \val(\gamma)\}
  \]
  satisfying $1 \le m \le \ell$, $\omega, \delta \in \pi \Lri_{\ell}$, and
  \begin{itemize}
    \item If $\bfG=\GL_3$ then $r=1$;

    \item If $\bfG=\GU_3$ then $r=\nonsq$,  $\gamma = \gamma^\circ$, $\omega = \omega^\circ$, $\alpha = \beta^\circ$, and $\delta = \delta^\circ$.

  \end{itemize}

\end{proposition}

The precise relations between the different representatives in Proposition~\ref{fibre.over.e} are involved and given explicitly in \cite[Theorems 2.20, 3.14]{AKOV2}. The incremental branching rules, however, are easier to describe.

\begin{theorem}\label{branching.e.GL} Let $A=A(\ell,\ell; \omega,\delta, 0, 0, 0)  \in \phi^{-1}_{\ell,1}(rE) \subset \mfg(\lri_\ell)$. The fibre $\phi^{-1}_{\ell+1,\ell}(A)$ is a torsor over $\mfh(\rf)^\Tr = \mfh(\Rf)^\Tr \cap \mfg(\rf)$, with
  $$h=h(\alpha_1,\alpha_2,\beta_1,\beta_2,\gamma) =\left[\begin{matrix} \alpha_1 &0  & 0 \\  \beta_1& \alpha_2 & 0 \\ \gamma &  \beta_2 & \alpha_1 \end{matrix}\right]  \in \mfh(\kk)^{\Tr}$$
  acting on elements in the fibre by
  $\tilde{X} \mapsto \tilde{X} +\pi^\ell h$.
  By fixing a base point $$ X_0=X(\ell+1,\ell+1; \wt{\omega},\wt{\delta}, 0, 0, 0)$$ with $\wt{\omega} \equiv_{\ell} \omega$ and $\wt{\delta} \equiv_{\ell} \delta$. We can describe the orbits in the fibre in terms of $\mfh(\rf)^{\Tr}$.

  \begin{itemize}[leftmargin=*]

    \item $\bfG=\GL_3$. In this case $\alpha_1,\alpha_2,\beta_1,\beta_2,\gamma \in \rf$, and above $A$  there are

          \begin{itemize}[leftmargin=*]
            \item[$\circ$] $q^2$ classes $X(\alpha_1,\alpha_2,0,0,0)$, all of them are singletons.
            \item[$\circ$] $q^2(q-1)$ classes represented by $X(\alpha_1,\alpha_2,*,*,\gamma)$, $\gamma \in \rf^\times$, $\alpha_1,\alpha_2 \in \rf$.
            \item[$\circ$] $q(q+1)$ classes  represented by $X(\alpha_1,\alpha_1,\beta_1,\beta_2,0)$,  $\alpha_1 \in \rf$, $[\beta_1:\beta_2] \in \mathbb{P}^1(\rf)$

          \end{itemize}

    \item $\bfG=\GU_3$. In this case $\alpha_1,\alpha_2,\beta_1,\beta_2,\gamma \in \Rf$ satisfying $\alpha_i +\alpha_i^\circ = 0$, $\gamma+\gamma^\circ=0$ and $\beta_1+\beta_2^\circ=0$. Above $A$ there are

          \begin{itemize}[leftmargin=*]
            \item[$\circ$] $q^2$ classes $X(\alpha_1,\alpha_2,0,0,0)$, all of them are singletons.
            \item[$\circ$] $q^2(q-1)$ classes represented by $X(\alpha_1,\alpha_2,*,*,\gamma)$, $\alpha_1,\alpha_2,\gamma \in \gu_1(\rf)$, $\gamma \ne 0$.
            \item[$\circ$] $q(q-1)$ classes  represented by $X(\alpha_1,\alpha_1,\beta_1,-\beta_1^\circ,0)$,  $\alpha_1 \in \gu_1(\rf)$, $[\beta_1:-\beta_1^\circ] \in \GU_1(\kk) \smallsetminus\{1,-1\}$. 

          \end{itemize}

  \end{itemize}

\end{theorem}

Theorem~\ref{branching.e.GL}, albeit with different representatives, is proved in \cite{AKOV2}: in \S2 for the general linear case and in \S3 for the unitary case. The branching rules, that is, the numerical classification, are encoded in Table 2.2 in {\em loc.\! cit.} in lines 13-17, with $\epsilon=1$ for the linear case and $\epsilon=-1$ for the unitary case. The representatives above are chosen as they are significantly better behaved for computations and for qualitative considerations. The orbit representatives for $\SL_3$ and $\SU_3$ for $p \ne 3$ are the traceless representatives above. The obstruction to transfer these results to the prime $3$ in these cases is the fact that the Killing form is degenerate.

\section{Proof of Theorem~\ref{thm:dim.reduction.decomposable}}\label{sec:proof.thm.D}

In this section we consider irreducible representations of $\bfG(\lri_\ell)$ lying over decomposable orbits and prove Theorem~\ref{thm:dim.reduction.decomposable}. Let $\xi =j(\xi_1,\xi_2)=\left(\begin{smallmatrix} \xi_1 & \\ & \xi_2\end{smallmatrix}\right) \in \mfg(\kk)$ be a representative of a decomposable orbit, that is, the characteristic polynomials of $\xi_1$ and $\xi_2$ are coprime. Let $\phi:\Mat_c(\Lri_m) \to \Mat_c(\Lri_1)$ be the reduction map, with a slight abuse of notation letting the level and the dimension vary with the context. By the first part of Proposition~\ref{prop:decomp.orbits}, for every $m \in \N$, every $\wt{\xi} \in \mfg(\lri_m)$ that lies over $\xi$ can be $\bfG(\lri_m)$-conjugated to a block matrix $j(\wt{\xi_1},\wt{\xi_2}) \in \mfg(\lri_m)$, giving rise to a bijection between orbits of the fibres $\phi^{-1}(\xi) \subset \mfg(\lri_m)$ over~$\xi$ to pair of orbits over~$\xi_1$ and $\xi_2$: 
\begin{equation}\label{eq:bijection.of.orbits}
\phi^{-1}\left(\Cent_{\bfG(\kk)}(\xi)\right) \backslash \phi^{-1}(\xi) \xleftrightarrow{\enspace 1:1 \enspace} \phi^{-1}\left(\Cent_{\bfG_1(\kk)}(\xi_1)\right) \backslash \phi^{-1}(\xi_1)\times \phi^{-1}\left(\Cent_{\bfG_2(\kk)}(\xi_2)\right) \backslash \phi^{-1}(\xi_2)  
\end{equation}

\subsection{Proof for $\ell=2m$ even} The $m^{\text{th}}$ congruence subgroup of $\bfG(\lri_\ell)$ is \[K^m=\exp(\pi^{m}\mfg(\lri_m)) \cong (\mfg(\lri_m),+).\] 
Similarly, write $K_i^m$ for $\exp(\pi^{m}\mfg_i(\lri_m))$, $i=1,2$. The second part of Proposition~\ref{prop:decomp.orbits} guarantees that the stabilisers on both sides of the correspondence are the same. Let $(\wt{\xi_1},\wt{\xi_2})$ be a representative of an orbit on the RHS of \eqref{eq:bijection.of.orbits} and $\wt{\xi}=j(\wt{\xi_1},\wt{\xi_2})$ be the chosen representative of the corresponding orbit on the LHS. By Clifford theory, and the bijection \eqref{eq:bijection.of.orbits}, proving that the bottom arrow in the following diagram is a bijection will imply that the upper arrow is a bijection:
\begin{equation}\label{diag.product}
\xymatrix{\Irr\left(\bfG(\lri_\ell) \mid j(\wt{\xi_1},\wt{\xi_2})\right) & \ar@{.>}[l]  \Irr\left(\bfG_1(\lri_\ell) \mid \wt{\xi_1}\right) \times \Irr\left(\bfG_2(\lri_\ell) \mid \wt{\xi_2}\right)\\
\Irr\left(\Stab_{\bfG(\lri_\ell)}(\wt{\xi}) \mid j(\wt{\xi_1},\wt{\xi_2})\right)  \ar[u]^{\Ind} & \ar[l]   \ar[u]^{\Ind} \Irr\left(\Stab_{\bfG_1(\lri_\ell)}(\wt{\xi_1}) \mid \wt{\xi_1}\right) \times \Irr\left(\Stab_{\bfG_2(\lri_\ell)}(\wt{\xi_2}) \mid \wt{\xi_2}\right).}
\end{equation}

Now, writing elements of $\mfg(\lri_m)$ as $2$-by-$2$ block matrices according to the blocks of $j(\wt{\xi_1},\wt{\xi_2})$, we see that the additive subgroup $\mathfrak{l}(\lri_m)=\mathfrak{l}_{12}(\lri_m) \oplus \mathfrak{l}_{21}(\lri_m)=\left\{\left(\begin{smallmatrix} 0 & * \\ * & 0 \end{smallmatrix}\right)\right\} \leq \mfg(\lri_m)$, consisting of the $(1,2)$ and $(2,1)$ blocks,  is in the kernel of the map $x \mapsto \psi(\mathrm{Tr}(\wt{\xi}x))$. Denoting $L^m:=\exp\left(\pi^m\mathfrak{l}(\lri_m)\right)$, we get an isomorphism  
\[
\Stab_{\bfG(\lri_\ell)}(\wt{\xi}) /L^m \cong \Stab_{\bfG_1(\lri_\ell)}(\wt{\xi_1}) \times \Stab_{\bfG_2(\lri_\ell)}(\wt{\xi_2}),
\]
which implies that the bottom map in \eqref{diag.product} is a bijection. 

\subsection{Proof for $\ell=2m+1$ odd}  The complication in the odd case comes from the fact that the largest abelian congruence kernel is $K^{m+1}$, which is smaller relative to the size of the group compared to even case. Similar considerations show that 
\begin{equation}\label{eq:stab.odd.decomp}
\Stab_{\bfG(\lri_\ell)}(\wt{\xi}) /K^{m+1} \cong \underbrace{L^{m}K^{m+1}/K^{m+1}}_{\cong \mathfrak{l}(\kk)} \rtimes \left(\Stab_{\bfG_1(\lri_\ell)}(\wt{\xi_1})/K_1^{m+1} \times \Stab_{\bfG_2(\lri_\ell)}(\wt{\xi_2})/K_2^{m+1}\right).
\end{equation}

The group $K^m$ is a subgroup of $\Stab_{\bfG(\lri_\ell)}(\wt{\xi}_m)$, because $[K^m,K^{m+1}]=\{1\}$, and therefore contains $K^{m+1} = \exp\left(\pi^{m+1}\mfg(\lri_{m})\right) \cong \left(\mfg(\lri_{m}),+\right)$ as a central subgroup. We have  
\[
K^m/K^{m+1} \cong \mfg(\kk) \cong  \mfg_1(\kk) \oplus \mfg_2(\kk) \oplus \mathfrak{l}_{12}(\kk) \oplus \mathfrak{l}_{21}(\kk). 
\]

The reduction of the element $\wt{\xi} = j(\wt{\xi_1},\wt{\xi_2}) \in \mfg(\lri_{m+1})$ to $\mfg(\lri_m)$ induces an alternating bilinear form on the quotient $\mfg(\kk)$ with radical $\mfg_1(\kk) \oplus \mfg_2(\kk)$, and $\wt{\xi}$  defines a linear extension to $K_1^mK_2^mK^{m+1}$. It follows that there is a unique irreducible representation $\rho({\wt{\xi}})$ of $K^m$ lying above that extension. Moreover, both $\mathfrak{l}_{12}(\kk)$ or $\mathfrak{l}_{21}(\kk)$ is a $\Stab_{\bfG(\lri_\ell)}(\wt{\xi}_m)$-stable Langrangian subspace, and one can explicitly construct the irreducible representation $\rho({\wt{\xi}})$ by extending the character $\wt{\xi}$ of $K_1^mK_2^mK^{m+1}$ trivially on $L_{12}^m=\exp(\pi^m\mathfrak{l}_{12}(\lri_{m+1}))$ and then induce irreducibly to $K^m$. The proof continues similar to the proof of the even case using the isomorphism~\eqref{eq:stab.odd.decomp}.

\section{Construction of non-regular nilpotent characters}
\label{sec:construction-of-non-regular-characters-G}

\subsection{Overview}
In this section we construct the  characters of $\bfG(\cO)$ lying above nilpotent orbits of rank one and level~${\ell-1}$ for $\bfG \in \{\GL_3, \GU_3$\}. These are, by definition, the primitive non-regular nilpotent characters of~$\bfG(\cO_{\ell})$. We assume throughout that $p=\mathrm{char}(\kk) > 3$. Let $\bfJ(\lri_\ell)=Z_{\bfG(\lri_\ell)}(E)$, where $E$ is defined in \eqref{def.E}. Recall that $\ldown = \lfloor \frac{\ell}{2} \rfloor$ and $\lup = \lceil \frac{\ell}{2} \rceil$. For $0 \leq m \leq \ell$, define  $\mdown = \lfloor \frac{\ell-m}{2} \rfloor$ and $\mup = \lceil \frac{\ell-m}{2} \rceil$. For $a,b \in \cO_\ell$, we write $a \equiv_{t} b$ rather than $a \equiv b \mod {(\pi^t)}$.

By the classification of the adjoint orbits lying over $E$, as given in Section~\ref{sec:adjoint-orbits-above-E}, for every non-regular nilpotent character $\chi$, there exists $A =  A(\ell, m;  \omega, \delta,  \alpha, \beta, \gamma) \in \mfh(\Lri_\ell)^\Tr \cap \mfg(\lri_\ell)$ of the form~(\ref{form.of.char}) with $m = \min \{\val(\alpha), \val(\beta), \val(\gamma)\} $ 
such that $\chi$ lies above $\psi_{A_\ldown}$. By Clifford theory $\Irr(\bfG(\lri_\ell) \mid \psi_{A_{\ldown}})$ is in bijection with $\Irr(\iga \mid \psi_{A_{\ldown}})$. Therefore, we focus on understanding the latter. Up to multiplication by a central character of $\bfG(\cO_\ell)$, we can omit~$\omega$ and shall assume from now onwards that  \[
  A_\ldown = rE +   \left[\begin{matrix} 0 &  &  \\    &\delta &  \\ \delta^2  &   & 0 \end{matrix}\right] +  \left[\begin{matrix} 0 &  &  \\   \alpha & 0 &  \\ \gamma &  \beta & 0 \end{matrix}\right],
\] 
where $\delta, \alpha, \beta, \gamma$ are considered modulo $\pi^{\ldown}$.

The following result, whose proof occupies the rest of this section, summarises the construction of all non-regular nilpotent characters of $\bfG(\cO_\ell)$. 

\begin{theorem}
\label{thm:construction}
  For every non-regular non-zero nilpotent character $\chi$ of $\bfG(\cO_\ell)$, there exists  $A  \in \mfh(\Lri_\ell)^\Tr \cap \mfg(\lri_\ell)$ of the form~(\ref{form.of.char}) such that $\chi$ lies above (an explicitly defined) character~$\psi_{A_\lup}$ of $\Stab_{\bfG(\lri_\ell)^{\ldown}}(\psi_{A_{\ldown}})$ which extends $\psi_{A_\ldown}$ and the following hold:
 \begin{itemize} 
 \item[(a)] For $m < \lup,$ there exists an explicitly constructed character $\sigma$ of  $\iga^{\mdown}$ such that $\sigma(1) = q^{2(\lup-\ldown) + (m_2-m_1)}$, $\sigma$ extends to its stabiliser $I(\sigma)$ in $\iga$, the extension induces irreducibly to $\bfG(\cO_\ell)$ and gives $\chi$. 
 \item[(b)] For $m \geq  \lup,$ there exists an explicitly constructed character $\sigma$ of $\iga^{\ldown}$ such that $\sigma(1) = q^{2(\lup-\ldown)}$ and $\sigma$ extends to $\iga$. This extension tensored with an irreducible character of $\bfJ(\cO_\ell)$ gives  $\chi$.  
 \end{itemize} 
\end{theorem}

To prove this result, we first describe the inertia group  $\iga$ in Section~\ref{sec:structure-of-stabilisers}. Next, we define an extension of $\psi_{A_\ldown}$ to a normal subgroup of $\iga$ in Section~\ref{psi.l.down}. Finally, we consider the case of $m<\lup$ and $m \geq \lup$ in Sections~\ref{subsec:m-smaller-ldown} and \ref{subsec:case-m-greater-ldown}, respectively, and complete the proof of the above theorem.  Before proceeding with the details, we record some of the basic properties of $m_1$ and $m_2$ with $m \leq \ldown$. We shall use the following lemma throughout without mentioning it explicitly.
\begin{lemma}
\label{lem:properties-m1-m2}
For $0 \leq m \leq \ldown$, $\mdown$ and $\mup$ satisfy the following: 
  \begin{enumerate}
    \item  $0 \leq \upu - \downu \leq 1$.
    \item $2 \upu + m \geq \ell $.
    \item If $\upu =\downu+1$ then $2 \downu + m = \ell-1$.
    \item If $m < \ldown$ then $2 \downu + \ldown \geq \ell$.
    \item For $m < \ldown$, $2 \upu \geq 2\downu \geq \lup$ and $4 \upu \geq 4\downu \geq \ell$.
    \item $m_2 + m \geq \lup$.
    \item For $m = \ldown$, $m_1 = \lfloor \frac{\lup}{2} \rfloor $ and $m_2 = \lceil \frac{\lup}{2} \rceil  $.
  \end{enumerate}
\end{lemma}
\begin{proof} (1) follows from the definition of $\mdown$ and $\mup$. For (2), $\mup + \mdown = \ell-m$ implies
  \[
    2 \mup + m \geq \mup + \mdown +  m = \ell.
  \]
For (3), the condition $\mup = \mdown + 1$ implies $2 m_1 + 1 = \ell -m$.  For (4), $2 \mdown + \ldown > 2 \mdown + m \geq  \ell-1$ implies $2 \mdown + \ldown \geq \ell$. (5) follows from (4). (6) follows from the following:
  \[
    \mup + m \geq  \frac{\ell + m}{2} \geq \frac{\ell}{2}. 
  \] (7) follows from the definition of $\mdown$ and $\mup$. \qedhere
\end{proof}

\subsection{Structure of stabilisers} 
\label{sec:structure-of-stabilisers}

 Let~$\wt{\bfJ}=\wt{\bfJ}_\delta$ be the sub-ring-scheme of $\Mat_3$ whose $R$-points for every $\lri_\ell$-algebra $R$ are 
\[
  \wt{\bfJ}(R) = \left\{
  \mat {a}{x}{z}{ \delta y}{b}{y}{\delta^2 z}{\delta x}{a} \mid  a,b,x,y,z \in R \right\}.
\]
Let $\wt{\bfH}_0$,  $\wt{\bfH}_\infty$ and $\wt{\bfH}_{\star}$ be the ring-schemes (all dependent on $\alpha,\beta,\gamma,\delta$) whose $R$-points are given by
\[
  \begin{split}
    &\wt{\bfH}_0(R)= \left\{
    \mat {a}{x}{z}{ \alpha z + \delta y}{b}{y}{\alpha x +\gamma z +\delta^2 z}{\beta z + \delta x}{a} \mid  \begin{matrix}  a,y,z \in R   \qquad   \qquad ~~  \\ \quad b = a+\delta z-\alpha^{-1}\gamma y  \\ x = \alpha^{-1}\beta y     \qquad \quad \qquad \end{matrix}  \right\}, \\
    & \wt{\bfH}_\infty(R)=  \left\{
    \mat {a}{x}{z}{ \alpha z + \delta y}{b}{y}{\alpha x +\gamma z +\delta^2 z}{\beta z + \delta x}{a} \mid  \begin{matrix}  a,y,z \in R   \qquad   \qquad ~~  \\ \quad b = a+\delta z-\beta^{-1}\gamma x  \\ y = \alpha \beta^{-1} x    \qquad \quad \qquad \end{matrix}  \right\}, \\
    & \wt{\bfH}_\star(R)=  \left\{
    \mat {a}{x}{z}{ \alpha z + \delta y}{b}{y}{\alpha x +\gamma z +\delta^2 z}{\beta z + \delta x}{a} \mid  \begin{matrix}  a,y,z \in R   \qquad   \qquad ~~  \\ x = \beta \gamma^{-1}(a-b + \delta z) \\ y = \alpha \gamma^{-1}(a-b + \delta z)   \end{matrix}  \right\}.\\
  \end{split}
\]

In our application the various ring-schemes will become relevant according to the parameters $\alpha, \beta$ and $\gamma$, and reflect the different  nature of solutions to equations \eqref{condition-on-centralizer01}-\eqref{condition-on-centralizer03}. Specifically, $\wt{\bfH}_0$ corresponds to the case $\val(\alpha)=m=\min\{\val(\alpha),\val(\beta),\val(\gamma)\}$, namely $\alpha$ takes a minimal valuation, $\wt{\bfH}_\infty$ corresponds to the case $\val(\beta)=m$, and $\wt{\bfH}_\star$ to the case $\val(\gamma)=m$. Note that under these circumstances a term like $\alpha^{-1}\beta$ for $\val(\alpha) \leq \val(\beta)$ is well-defined in $\lri_{\ell-m}$, even if $\alpha$ is not invertible. To make it well-defined in $\lri_\ell$ one can lift $\alpha$ and $\beta$ to $\lri_{\ell+m}$. Different lifts may result in different ring-schemes but that will not affect our applications.
\smallskip

For $\wt{\bfU} \in  \{\wt{\bfJ}, \wt{\bfH}_0, \wt{\bfH}_\infty, \wt{\bfH}_\star\}$, let ${\bfU}=\wt{\bfU} \cap \bfG$ be the corresponding group sub-scheme of $\bfG$ and let $\mfj=\wt{\bfJ} \cap \mfg$,  $\mfh_0=\wt{\bfH}_0 \cap \mfg$, $\mfh_\infty=\wt{\bfH}_\infty \cap \mfg$ and $\mfh_\star=\wt{\bfH}_\star \cap \mfg$ be the corresponding Lie sub-scheme of $\mfg$. Let $\mathcal{M}_{A,k}$ denote the set of  matrices of the form
\begin{equation}\label{X}
  X = \mat {a}{x}{z}{\alpha z+ \delta y}{b}{y}{\alpha x+ \gamma z + \delta^2 z}{\beta z + \delta x}{a}  \in \Mat_3(\Lri_\ell),
\end{equation}
satisfying
\begin{eqnarray}
  \label{condition-on-centralizer01}
  \beta y & \equiv_{k} & \alpha x \\
  \gamma y  & \equiv_{k} &  \alpha (a-b + \delta z) \label{condition-on-centralizer02}\\
  \gamma x & \equiv_{k} &  \beta (a-b + \delta z).
  \label{condition-on-centralizer03}
\end{eqnarray}

We now compute the stabilisers under the adjoint action of $\bfG(\lri_\ell)$ and its congruence subgroups on the image $A_k \in \mfg(\lri_k)$ of an element of the form $A=A(\ell, m; \omega, \delta,  \alpha,\beta,\gamma) \in \mfg(\lri_\ell)$ for $k \le \ell$; see~\eqref{form.of.char}. For $j \le k \le \ell,$ let $\mathfrak{s}^j_{A,k}\subset \mfg(\lri_\ell)$ denote the $j$-th congruence  annihilator in $\mfg(\lri_\ell)$ of $A_k$, and similarly let  $\mathcal{S}^j_{A,k} \subset \bfG(\lri_\ell)$ be   the $j$-th congruence  subgroup of the stabiliser  in $\bfG(\lri_\ell)$ of $A_k$. Also, let $\mathcal{M}^j_{A,k}=   \mathcal{M}_{A,k} \cap \mfg(\lri_\ell)^j$.  
\smallskip

\begin{proposition}
  \label{Prop:S-basic} With the above notation, let $\bfH$ be one of the group schemes $\bfH_0$, $\bfH_\infty$ or $\bfH_\star$ according to the minimal valuation of $\alpha, \beta$ and $\gamma$. Then the following hold:
  \begin{enumerate}

    \item $\mcM^j_{A,k}$ is an $\Lri_\ell$-subalgebra for $j \ge \min\{\frac{\ell-m}{2}, \ell-k\}$.
    \item $\mathfrak{s}^j_{A,k}=M^j_{A,k} \cap \mfg(\lri_\ell)+\mfg(\lri_\ell)^{\max\{j,k\}}=\mfh(\lri_\ell)^j + \mfj(\lri_\ell)^{\max\{j,k-m\}} + \mfg(\lri_\ell)^{\max\{j,k\}}$.

    \item $\mcS^j_{A,k}=\left(\mcM^j_{A,k} \cap \bfG(\lri_\ell)\right)\bfG(\lri_\ell)^{\max\{j,k\}}=\bfH(\lri_\ell)^j \bfJ(\lri_\ell)^{\max\{j,k-m\}} \bfG(\lri_\ell)^{\max \{ j,k\}}$.

    \item $\bfH(\lri_\ell)^j \subset \mcS^j_{A,\ell}$ for all \,\, $0 \leq j \leq \ell$.

  \end{enumerate}

\end{proposition}

\begin{proof}
  \begin{enumerate}
    \item[(1)]

      Clearly $\mcM^j_{A,k}$ is closed under addition for every $0 \le j,k
        \le \ell$. Multiplying two matrices $X_1,X_2 \in \mcM^j_{A,k}$ we check that the entries $(2,1)$, $(3,1)$ and $(3,2)$ of $X=X_1X_2$ conform to the pattern in \eqref{X} and that equations  \eqref{condition-on-centralizer01}, \eqref{condition-on-centralizer02} and \eqref{condition-on-centralizer03} hold. Indeed, under the assumption that $X_i \in \mcM^j_{A,k}$ we have  $x_i,y_i,z_i,a_i-1,b_i-1 \in \pi^j\Lri_\ell$, hence
      \[
        \begin{split}
          (a) ~  &X(2,1)-\alpha X(1,2)-\delta X(2,3)= z_2(\gamma y_1-\alpha \delta z_1 - \alpha a_1 + \alpha b_1) + (x_2y_1 - x_1y_2)\alpha \equiv_\ell 0  \\
          (b) ~ &X(3,1)-(\delta^2 +\gamma)X(1,3)- \alpha X(1,2)  = x_1(\alpha(\delta z_2 + a_2 - b_2)-\gamma y_2)  + z_1 \delta(y_2\beta  - \alpha x_2)  \equiv_\ell 0 \\
          (c) ~  &X(3,2)-\beta X(1, 3) - \delta X(1,2) =  z_1(\gamma x_2-\delta \beta z_2 - \beta a_2  + \beta b_2) - \beta x_1 y_2 + \alpha x_1x_2  \equiv_\ell 0 \\
          (d) ~  &\alpha X(1, 2) - \beta X(2,3) =   (\delta \alpha x_2 z_1 +\alpha a_1x_2) -  \beta b_1y_2      - (\delta \beta y_1 z_2 + \beta a_2y_1 - \alpha b_2 x_1 )\\
          & \qquad \qquad \qquad \qquad \qquad  \qquad \equiv_k  x_2( \alpha \delta z_1 + \alpha a_1  - \alpha b_1)      - y_1(\delta \beta z_2 + \beta a_2 - \beta b_2)  \equiv_k 0 \\
          (e) ~ &  \gamma X(2,3) - \alpha(\delta X(1, 3) + X(1, 1) - X(2, 2) ) \equiv_k 0 \\
          (f) ~ &  \gamma X(1, 2) - \beta(\delta X(1, 3) + X(1, 1) - X(2, 2) )\equiv_k 0,
        \end{split}
      \]
      
\noindent where we repeatedly used $j \ge \min\{\frac{\ell-m}{2}, \ell-k\}$, and the relations \eqref{condition-on-centralizer01}, \eqref{condition-on-centralizer02} and \eqref{condition-on-centralizer03} for $X_i$ and omitted the details in the last two cases which are similar.

    \item[(2)] To prove the first equality, we note that a matrix $X$ commutes with $A$ modulo $\pi^k$, i.e. 
      \[
        XA={\mat a  x z {x'} b y {z'} {y'} {c}}\newA \equiv_k \newA  \mat a  x z {x'} b y {z'} {y'} {c}=AX
      \]
      if and only if
      \begin{eqnarray}
        z' & \equiv_k & \alpha x + \gamma z+ \delta^2 z \\
        y' & \equiv_k & \beta z +  \delta x\\
        c & \equiv_k & a \\
        x' & \equiv_k & \alpha z +  \delta y \\
        \beta y & \equiv_k & \alpha x \\
        \gamma y - \alpha \delta  z& =_k & \alpha (a-b) \\
        \gamma x - \beta \delta z & \equiv_k & \beta (a-b)
      \end{eqnarray}
      It follows from these relations that $\mathfrak{s}^j_{A,k}=\mcM^j_{A,k} \cap \mfg(\lri_\ell)+\mfg(\lri_\ell)^{\max\{j,k\}}$. The second equality follows by observing that
      $$ \mcM^j_{A,k} \cap \mfg(\lri_\ell) = \mfh(\lri_\ell)^j + \mfj(\lri_\ell)^{\max\{j,k-m\}} + \mfg(\lri_\ell)^{\max\{j,k\}}. $$
    \item[(3)] This follows from (2) and from the definition of $\mcS_{A, k}^j$ and $\mfs_{A,k}^j$.
    \item[(4)] This follows directly from (3). \qedhere
  \end{enumerate}
\end{proof}
Recall that $\iga$ denote the stabiliser (or the inertia group) of $\psi_{A_{\ldown}}$ in $\bfG(\lri_\ell)$. 
\begin{proposition}
  \label{prop:psiA-stabiliser}
  The stabiliser $\iga$ of $\psi_{A_{\ldown}}$ satisfies the following:
  \begin{itemize}
    \item[(a)] $\iga = \begin{cases}  \bfH(\lri_\ell) \bfJ(\lri_\ell)^{\ldown-m} \bfG(\lri_\ell)^{\ldown} & \mathrm{for} \,\, m < \ldown, \\ 
    \bfJ(\lri_\ell) \bfG(\lri_\ell)^{\ldown} & \mathrm{for}\,\, m \geq \ldown. 
    \end{cases} $
    \item[(b)] $\iga^\downu =  \bfJ(\lri_\ell)^{\mdown } \bfG(\lri_\ell)^{\ldown}$.

  \end{itemize}
  \end{proposition}

\begin{proof} By using the definition of $\psi_{A_\ldown},$ we obtain $\iga = \mcS_{A, \ldown}^0$ and $\iga^{\mdown} = \mcS_{A, \ldown}^{\mdown}$. Therefore (a) and (b) follow from $\mup \geq \mdown \geq \ldown-m$, $\mdown + m \geq \ldown$ and Proposition~\ref{Prop:S-basic}. 
\end{proof}

\subsection{Extension of $\psi_{A_\ldown}$}\label{psi.l.down} In this section we analyse a specific normal subgroup of $\iga$ and extend $\psi_{A_\ldown}$ to this subgroup. 

\begin{lemma}
    \begin{enumerate}
         \item[(a)] The group $\bfJ(\lri_\ell)^{\mup} \bfG(\lri_\ell)^{\lup}$ is a normal subgroup of $\iga$.
    \item[(b)] The quotient  $\iga^{\downu}/   ( \bfJ(\lri_\ell)^{\mup} \bfG(\lri_\ell)^{\lup})$ is an abelian group.
    \end{enumerate}
\end{lemma}
\begin{proof}
For~(a), by using $m + \mup \geq \lup$, we note that
  $$  xyx^{-1}  \in \mfj(\cO_\ell)^{\mup} + \mfg(\cO_\ell)^{\lup} $$ for all $x \in \bfH(\cO_\ell)$ and $y \in \mfj(\cO_\ell)^{\mup} + \mfg(\cO_\ell)^{\lup}.$ Therefore
  $(a)$ follows. To prove (b), we note that $[g,h] \in \bfG(\lri_\ell)^{2 \downu} \subseteq \bfG(\lri_\ell)^{\lup}$ for all $g, h \in \iga^{\mdown}$. Therefore the quotient group $\iga/(\bfJ(\lri_\ell)^{\mup} \bfG(\lri_\ell)^{\lup})$ is abelian.
\end{proof}

\begin{proposition}
  \label{prop:Stab-psi-B-m1}  The map $\psi_{A_{\ell - \mup}}: \bfJ(\lri_\ell)^\upu \, \bfG(\lri_\ell)^{\lup} \rightarrow \mathbb C^\times$
  defined by
  \[
    \psi_{A_{\ell - \mup}}(g) = \psi(\tr (A_{\ell - \mup} \log(g))),
  \]
  for all $g \in \bfJ(\lri_\ell)^\upu  \bfG(\lri_\ell)^{\lup}$ is a well-defined character  that extends $\psi_{A_\ldown}.$
\end{proposition}
\begin{proof} This proof is straightforward using the Baker-Campbell-Hausdorff formula:
  \[
    \begin{split}
      \psi_{A_{\ell-\mup}}(\exp(X)\exp(Y))&=\psi_{A_{\ell-\mup}}(\exp(X+Y+Z))=\psi(\tr(A_{\ell-\mup}(X+Y+Z)) \\&=\psi(\tr(A_{\ell-\mup}X+A_{\ell-\mup}Y)=\psi_{A_{\ell-\mup}}(\exp(X))\psi_{A_{\ell-\mup}}(\exp(Y))
    \end{split}
  \]
  because $Z \in  \ker(\tr(A_{\ell - \mup} \cdot))$. Indeed $\mfj^{\mup}$ commutes with $A_{\ell - \mup}$ and  $Z \in ([\mfj^{\mup}, \mfj^{\mup}] + [\mfj^{\mup}, \mfg^{\lup}] + [\mfg^{\lup}, \mfg^{\lup}])$ by Baker-Campbell-Hausdorff formula. 
\end{proof}
Note that $(\bfJ(\cO_\ell)^{\mup} \bfG(\cO_\ell)^\lup)/\bfG(\cO_\ell)^\lup$ is abelian, hence every character of $\bfJ(\cO_\ell)^{\mup} \bfG(\cO_\ell)^\lup$ lying above $\psi_{A_{\ldown}}$ is of the form $\psi_{A_{\ell-\mup}}$ for some $A$ lying above $\psi_{A_{\ldown}}.$ 
We now proceed to prove Theorem~\ref{thm:construction}.

\subsection{Proof of Theorem~\ref{thm:construction}(a)} 
\label{subsec:m-smaller-ldown}

In this section, we prove Theorem~\ref{thm:construction}(a). Hence we assume  $m < \lup$ throughout this section. 
 We work here with arbitrary $A_{\ell-\mup}$ and focus on studying $\Irr(\iga \mid \psi_{A_{\ell-\mup}})$. This, by Clifford theory, completely describes $\Irr(\iga \mid \psi_{A_{\ldown}})$.

\begin{proposition}
  \label{prop:Stab-psi-B}
  The stabiliser $\Stab_{\iga}(\psi_{A_{\ell-\mup}})$ of $\psi_{A_{\ell-\mup}}$ satisfies
  \[
    \Stab_{\iga}(\psi_{A_{\ell-\mup}}) = \bfH(\lri_\ell) \bfJ(\lri_\ell)^{\mdown} \bfG(\lri_\ell)^{\ldown}.
  \]
\end{proposition}

\begin{proof}  An element $g  \in \iga$ stabilises $\psi_{A_{\ell-\mup}}$ if and only if
  \begin{equation}\label{eqn:stab-psi-B}
    \tr \left(A \cdot \left(g Y g^{-1} -Y\right)\right) =0, \quad \forall \,  Y \in \log \left(\bfJ(\lri_\ell)^{\mup} \bfG(\lri_\ell)^{\lup}\right)
  \end{equation}

  We first check that each of the subgroups on the RHS is contained in the LHS. Indeed,
  \begin{itemize}

    \item By Part (4) of Prop \ref{Prop:S-basic}, $\bfH(\lri_\ell) \subset \Stab_{\iga}(\psi_{A_{\ell-\mup}})$.

    \item To show that $\bfG(\lri_\ell)^{\ldown} \subset \iga $,
          \begin{eqnarray}
            (\psi_{A_{\ell-\mup}})^{\exp\left(\pi^{\ldown }X\right)} (\exp(\pi^\upu Y))  = &  \psi_{A_{\ell-\mup}}((\exp(\pi^\upu Y)) \psi_{A_{\ell-\mup}}(\exp(\pi^{\ldown + \upu } (XY-YX))). \nonumber
          \end{eqnarray}
          We use $\ldown + \upu \geq \lup$ and $\psi_{A_{\ell-\mup}}(\exp\pi^{\ldown + \upu } (XY-YX)) = \psi_A(\exp\pi^{\ldown + \upu } (XY-YX))  = 1$ to deduce that $\bfG(\lri_\ell)^{\ldown}\subseteq \mathrm{Stab}_{\iga}(\psi_{A_{\ell-\mup}})$.

    \item Finally, recalling that $\iga=\bfH(\lri_\ell) \bfJ(\lri_\ell)^{\ldown-m} \bfG(\lri_\ell)^{\ldown}$ it now remains to show that $\bfJ(\lri_\ell)^{\mdown}$ is contained in the LHS. Since~$\bfJ(\lri_\ell)^{\mdown}$ stabilises the restriction of the character $\psi_{A_{\ell- \mup}}$ to~$\bfG(\lri_\ell)^{\lup}$ it is enough to check that it stabilises its restriction to~$\bfJ(\lri_\ell)^{\mup}$. We therefore look at \eqref{eqn:stab-psi-B} with $g=\exp(\pi^{\mdown}X) \in \bfJ(\lri_\ell)^{\mdown}$ and $\pi^{\mup}Y \in \log \left( \bfJ(\lri_\ell)^{\mup} \right)$. Now, using the fact that $\mup+\mdown=\ell-m$, we have
          \[
            \tr \left(\pi^{\ell-m}A [X,Y] \right) = \tr \left(\pi^{\ell-m}X [A,Y] \right) = 0.
          \]
          This justifies that the LHS of (\ref{eqn:stab-psi-B}) is contained in the RHS.
    \item For the reverse inclusion, it is enough to
          prove that $$\bfJ(\cO_\ell)^{\ldown-m} \cap \Stab_{\iga}(\psi_{A_{\ell-\mup}}) \subseteq \bfH(\cO_\ell)^{\ldown-m} \bfJ(\cO_\ell)^{\mdown}\bfG(\cO_\ell)^{\ldown}.$$

          Consider $g = \exp(\pi^{\ldown-m}X) \in \bfJ(\cO_\ell)^{\ldown-m}$. Then $g$ stabilises $\psi(A_{\ell - \mup})$ if and only if
          \[
            \tr \left(\pi^{\mup}A \cdot \left(g Y g^{-1} -Y\right)\right) = 0
          \]
          for every $\pi^{\mup}Y \in \log \left( \bfJ(\lri_\ell)^{\mup} \right)$. We will consider various possibilities of $Y$ to obtain that $g \in \bfH(\cO_\ell)^{\ldown-m} \bfJ(\cO_\ell)^{\mdown}\bfG(\cO_\ell)^{\ldown}$.

          Consider  $Y_1 = \begin{bmatrix} 0&  1&  0 \\ 0 &  0&  0 \\ 0 &  \delta &  0 \end{bmatrix} $,  $Y_2 = \begin{bmatrix} 0&  0&  0 \\ \delta &  0&  1 \\ 0&  0 &  0 \end{bmatrix} $, and $Y_3 = \begin{bmatrix} 0&  0&  0 \\ 0&  1&  0 \\ 0 &  0 &  0 \end{bmatrix}. $ 
          We note that either $\pi^{\mup}Y_i$'s or their linear combinations give three linearly independent vectors of $\log \left(\bfJ(\cO_\ell)^{\mup}\right)$.We use Maple to compute $\tr \left(\pi^{\mup}A \cdot \left(g Y g^{-1} -Y\right)\right) $ for $g = \exp(\pi^{\ldown-m}X)$ with $X = \begin{bmatrix} a&  x&  z \\ \delta y&  b&  y \\ \delta^2 z &  \delta x &  a \end{bmatrix} $ and above choices of $Y$ to obtain condition on $g$ such that $g$ stabilises $\psi(A_{\ell - \mup})$. This leads to four relations, of which one is redundant, giving rise to the following three relations:
          \begin{equation}\label{eqn.stab.rel}
            \begin{split}
              &[ \gamma y- \alpha(a-b+\delta z) ] +\pi^mQ_1(x,y,z,a-1,b-1)\equiv_{\ell - \mup} 0 \\
              &[\gamma x- \beta(a-b+\delta z)] + \pi^mQ_2(x,y,z,a-1,b-1) \equiv_{\ell - \mup} 0 \\
              &[\alpha x - \beta  y] +\pi^mQ_3(x,y,z,a-1,b-1)\equiv_{\ell - \mup} 0,
            \end{split}
          \end{equation}

          where the $Q_1, Q_2$ and $Q_3$ are quadratic polynomials. We claim that $X=U+V+W$ with  $U \in \bfH(\lri_\ell)$, $V \in \bfJ(\lri_\ell)^{\mdown}$ and $W \in \bfG(\lri_\ell)^{\ldown}$. The proof is case by case, depending on the precise form of $\bfH$, using a reverse induction on $i \in \{\mdown,\ldots,\ldown-m\}$ with $X \in  \mfj(\lri_\ell)^i$. We give explicit proof for $\bfH_0$, the other two cases are similar. Indeed, given $X \in \mfj(\lri_\ell)^i$ as above we define
          \[
            X' = \begin{bmatrix} a&  x_1&  z \\ \alpha z +\delta y&  b_1&  y \\ \delta^2 z +\alpha x_1 +\beta y& \beta z + \delta x_1 &  a \end{bmatrix}  \in \bfH_0(\lri_\ell)
          \]
          with $x_1 := \alpha^{-1} \beta y$ and $b_1:= a + \delta z - \alpha^{-1} \gamma y$. Then
          using \eqref{eqn.stab.rel} observe that
          \[
            X-X'= \underbrace{\begin{bmatrix} 0 &  \alpha^{-1}(\alpha x-\beta y)&  0 \\  0 & b-a-\delta z +\alpha^{-1}\gamma y &  0 \\  0 & \delta \alpha^{-1} (\alpha x -\beta y) &  0 \end{bmatrix}}_{\in \mfj(\lri_\ell)^{i+1}} - \underbrace{\begin{bmatrix} 0 &  0 &  0 \\ \alpha z & 0 &  0 \\ 2\beta y& \beta z  &  0 \end{bmatrix}}_{\in \mfg(\lri_\ell)^{\ldown}}.
            \qedhere \]
  \end{itemize}

\end{proof}

Let $\Gamma$ denote the quotient
$\bfJ(\lri_\ell)^{\mdown}\bfG(\lri_\ell)^{\ldown} /\bfJ(\lri_\ell)^{\mup}\bfG(\lri_\ell)^{\lup} \cong \rf^{5(\mup-\mdown)+4(\lup-\ldown)}$,
and define a bilinear form $\Gamma  \times \Gamma \longrightarrow \C^\times$
\[
  \mcB_A\left(\exp(\pi^{\mdown}X),\exp(\pi^{\mdown}Y)\right):=\psi_{A}(\pi^{2\mdown}[X,Y]).
\]
The bilinear form $\mcB_A$ is well-defined because of the following:
\begin{gather*}
  [\exp(\pi^{\mdown}X),\exp(\pi^{\mdown}Y)] = \exp(\pi^{2\mdown}[X,Y] + \frac{1}{2}\pi^{3\mdown}[X, [X,Y]] + \frac{1}{3!} \pi^{4\mdown}[X,[X,[X,Y]]]),\\
  \psi_A([\exp(\pi^{\mdown}X),\exp(\pi^{\mdown}Y)] ) = \psi_A(\pi^{2\mdown}[X,Y] + \frac{1}{2}\pi^{3\mdown}[X, [X,Y]] + \frac{1}{3!}\pi^{4\mdown}[X,[X,[X,Y]]]),\\
  \psi_A(\frac{1}{2}\pi^{3\mdown}[X, [X,Y]]) = 1 \text{ and }
  \psi_A( \frac{1}{3!}\pi^{4\mdown}[X,[X,[X,Y]]]) = 1. 
\end{gather*}

\begin{proposition} The radical of $\mcB_A$ in $\Gamma$ is isomorphic to $\rf^{3(m_2-m_1)}$. Its inverse image in $\bfG(\lri_\ell)$ is given by
  \[
    \widetilde{R_A}=
    \bfH(\lri_\ell)^{\mdown}\bfJ(\lri_\ell)^{\mup}\bfG(\lri_\ell)^{\lup}.
  \]
\end{proposition}

\begin{proof} We first note that $\bfG(\lri_\ell)^{\ldown} \cap \wt{R_A} = \bfJ(\lri_\ell)^{\ldown}\bfG(\lri_\ell)^{\lup}$. Indeed, as we only need to consider the case $\ldown < \lup$, and noting that in such case $2\ldown=\ell-1$, we get that $\tr(\pi^{\ell-1}A[X,Y])=0$ for all $\ol{X} \in \mfg(\rf)$ implies that $\ol{Y} \in \mfj(\rf)$.
  Second, we need to show that $\bfJ(\lri_\ell)^{\mdown} \cap \wt{R_A} = \bfH(\lri_\ell)^{\mdown}\bfJ(\lri_\ell)^{\mup}$. A similar computation to the one given in the proof of Proposition~\ref{prop:Stab-psi-B} gives an even simpler version of Equations \eqref{eqn.stab.rel} without the quadratic terms hence the assertion follows. 
\end{proof}

\begin{proof}[Proof of Theorem~\ref{thm:construction}(a)] The construction from this point onwards is along the same lines as the construction in \cite{KOS}. The character $\psi_{A_{\ell - \mup}}$ extends to $\widetilde{R_A}$ say $\psi_{A_{\ell - \mdown}}$. We choose a maximal isotropic subspace in $\Gamma$ with respect to $\mcB_A$ which is stable under the Sylow-$p$ subgroup of $\bfH(\lri_\ell)$ (recalling that the latter is abelian). We then extend the character $\psi_{A_{\ell-\mdown}}$ to the inverse image in $\bfG(\lri_\ell)$ of that maximal isotropic subspace using $\psi_A$ (see \cite[Section~3.4]{KOS} for more details) such that the extension is $\bfH(\lri_\ell)$ stable. Therefore $\bfH(\lri_\ell)$ stabilises the unique irreducible $\sigma$ of $\bfJ(\lri_\ell)^{\mdown}\bfG(\lri_\ell)^{\ldown}$ lying above that extension. Since $\bfH(\lri_\ell)$ is abelian, we get an irreducible representation of $\iga$ lying above $\psi_{A_{\ldown}}$ which induces irreducibly to $\bfG(\lri_\ell)$.
\end{proof}
\subsection{Proof of Theorem~\ref{thm:construction}(b)}
\label{subsec:case-m-greater-ldown}
This section aims to prove Theorem~\ref{thm:construction}(b). Note that, for this case,  $m \geq  \lup$ and $\ell - \mup \geq \lup$. Therefore by Proposition~\ref{prop:Stab-psi-B-m1},  \[
    \psi_{A_{\lup}}(g) = \psi(\tr (A_{\lup} \log(g))),
  \]
  for all $g \in \bfJ(\lri_\ell)^\ldown \bfG(\lri_\ell)^{\lup}$ is a well-defined character of $\bfJ(\lri_\ell)^\ldown \bfG(\lri_\ell)^{\lup}$ that extends $\psi_{A_\ldown}.$ We focus to study $\Irr(\iga \mid \psi_{A_{\lup}})$. As earlier, this and Clifford theory  describe $\Irr(\iga \mid \psi_{A_{\ldown}}).$ Since $m \geq \lup$, we have 
 
\[
  A_{\lup} = rE + \omega \mathrm{I} +   \left[\begin{matrix} 0 &  &  \\    &\delta_1 &  \\ \delta_1^2  &   & 0 \end{matrix}\right]
\] 
such that $\omega \equiv_{\ldown} 0 $ and $\delta_1 \equiv_{\ldown} \delta.$
The following describes the stabiliser of $\psi_{A_{\lup}}$ in $\iga$. 
\begin{proposition}
  \label{prop:Stab-psi-B1}
  The stabiliser $\Stab_{\iga}(\psi_{A_{\lup}})$ of $\psi_{A_{\lup}}$ satisfies
  \[
    \Stab_{\iga}(\psi_{A_{\lup}}) = \bfJ(\lri_\ell) \bfG(\lri_\ell)^{\ldown}.
  \]
\end{proposition}
\begin{proof} The proof follows the same lines as that of Proposition~\ref{prop:Stab-psi-B} and we omit the details. 
\end{proof}
Our next claim is that $\psi_{A_{\lup}}|_{\bfJ(\lri_\ell)^{\lup}}$ extends to $\bfJ(\lri_\ell)$. To prove this, we first determine the commutator subgroup of $\bfJ(\cO_\ell)$. 

\begin{lemma}
\label{lem:comm-subgp-J}
The map $\bfJ(\lri_\ell)$ to $\bfG_1(\lri_\ell) \times \bfG_1(\lri_\ell)$, where $\bfG_1$ is either $\GL_1$ or $\GU_1$ depending on whether the group is linear or unitary, respectively, given by 
  \begin{equation}\label{twist.det}
    \Delta: \left[\begin{matrix} a & x & z \\ \delta y& b & y \\\delta^2z & \delta x& a \end{matrix}\right] \mapsto (a-\delta z, b(a+\delta z)-2\delta xy)
  \end{equation}
  is a surjective homomorphism. Its kernel
  \[
    \left\{ \left[\begin{matrix} 1+\delta z & x & z \\ \delta y& \frac{1+2\delta xy}{1+2\delta z} & y \\\delta^2 z & \delta  x& 1+ \delta z \end{matrix}\right] \mid x,y,z \in \lri_\ell \right\} \cap \bfJ(\lri_\ell),
  \]
  is the commutator subgroup $[\bfJ(\lri_\ell),\bfJ(\lri_\ell)]$. 
\end{lemma}

\begin{proof} It is immediate to check that this is a surjective homomorphism. The kernel contains the commutator subgroup because the quotient is abelian. Conversely, the matrices 
  \[
    \left[\begin{matrix} 1+\delta z & x & z \\ \delta y& \frac{1+2\delta xy}{1+2\delta z} & y \\\delta^2 z & \delta x& 1+\delta z \end{matrix}\right]
  \]
 are in the commutator subgroup and generate the kernel.
\end{proof}

\begin{proposition}
\label{prop:comm.sbgp-extension}
It follows that

  \begin{enumerate}
    \item[(i)] All linear characters of $\bfJ(\lri_\ell)$ are given by post-composing $\Delta$ with a character of  $\bfG_1(\lri_\ell) \times \bfG_1(\lri_\ell)$; and that

    \item[(ii)]
      \[
        [\bfJ(\lri_\ell),\bfJ(\lri_\ell)] \cap \bfJ(\lri_\ell)^{\lup} = \left\{\left[\begin{matrix} 1+\delta \pi^{\lup}z  &\pi^{\lup} x  & \pi^{\lup} z \\ \delta \pi^{\lup} y & 1-2\delta \pi^{\lup} z &  \pi^{\lup} y\\ \delta^2 \pi^{\lup}z & \delta \pi^{\lup}x & 1+\pi^{\lup}z \end{matrix}\right] \mid x,y,z \in \lri_\ell \right\} \cap \bfJ(\lri_\ell).
      \] 
      \item[(iii)] The character $\psi_{A_{\lup}}|_{\bfJ(\lri_\ell)^{\lup}}$ extends to $\bfJ(\lri_\ell)$.
  \end{enumerate}
\end{proposition} 

\begin{proof} Here (i) and (ii) follow from Lemma~\ref{lem:comm-subgp-J}, and (iii) follows from 
\[
\psi_{A_{\lup}} ([\bfJ(\lri_\ell),\bfJ(\lri_\ell)] \cap \bfJ(\lri_\ell)^{\lup}) = 1.    \qedhere
\]
\end{proof}
\begin{proof}[Proof of Theorem~\ref{thm:construction}(b)] For $\ell$ even, by virtue of Propositions~\ref{prop:Stab-psi-B} and ~\ref{prop:comm.sbgp-extension}, the character $\psi_{A_\lup}$ extends to $\iga$. Hence the result follows. 

Assume now that $\ell$ is odd.  
Let $\Gamma_1$ denote the quotient
$\bfG(\lri_\ell)^{\ldown} /(\bfJ(\lri_\ell)^\ldown \bfG(\lri_\ell)^{\lup}) \cong \rf^{4(\lup-\ldown)}$. Define a bilinear form $\mcB'_A: \Gamma_1  \times \Gamma_1 \longrightarrow \C^\times$ by 
\[
  \mcB'_A\left(\exp(\pi^{\ldown}X),\exp(\pi^{\ldown}Y)\right):=\psi_{A_{\lup}}(\pi^{2\ldown}[X,Y]).
\]
It is easily seen that $\mcB'_A$ is a well-defined non-degenerate bilinear form. The construction of $\Irr(\iga \mid \psi_{A_{\ldown}})$ from this point onwards follows along the same lines as the construction in~\cite{KOS}. We choose a maximal isotropic subspace in $\Gamma_1$ for $\mcB'_A$, which is stable under the (unique) {$p$-Sylow} subgroup of $\bfJ(\cO_\ell)$. We then extend the character $\psi_{A_{\lup}}$ to the inverse image in $\bfG(\lri_\ell)^{\ldown}$ of that maximal isotropic subspace using $\psi_A$ (see \cite[Section~3.4]{KOS} for more details). Since the $p$-Sylow subgroup of $\bfJ(\cO_\ell)$ stabilises the extension and therefore stabilises the unique irreducible representation $\sigma$ of $\bfG(\lri_\ell)^{\ldown}$ lying above $\psi_{A_{\lup}}$. The character $\psi_{A_\lup}$ extends to $\bfJ(\cO_\ell)$ by Proposition~\ref{prop:comm.sbgp-extension}. Therefore $\sigma$ extends to an irreducible representation $\tilde\sigma$ of $\iga$. By Clifford's theory, every irreducible representation of $\iga$ lying above $\psi_{A_\lup}$ is of the form $\tilde{\sigma} \otimes \chi,$ where $\chi$ is obtained from an irreducible representation of $\bfJ(\cO_\ell)$ by inflation. This proves Theorem~\ref{thm:construction}(b). 
\end{proof}

\subsection{The representation zeta function $\zeta_{\bfG(\lri_\ell) | E}(s)$}\label{eq:explicit.zeta.E} Using the  construction above we obtain the following explicit formula  for $\zeta_{\bfG(\lri_\ell) | E}(s)$.
\[
\begin{split}
\zeta_{\bfG(\lri_\ell) | E}(s) = &\sum_{m=1}^{\ldown} \sum_{A = A(\lup,m; \omega,\delta, \alpha, \beta, \gamma ) \in \mfg(\lri_{\lup})} [\bfG(\lri_\ell):\iga]^{-1} \\& \times \sum_{{A_{\ell - \mup}} | A_\ldown}  [\iga:\Stab_{\iga}(\psi_{A_{\ell - \mup}})]^{-1}q^{-2(\mup-\mdown)}q^{-(2(\lup - \ldown)+(\mup-\mdown))s}  \\
    & \times [\Stab_{\iga}(\psi_{A_{\ell - \mup}}): \bfG(\lri_\ell)^{\ldown}]  [\bfG(\lri_\ell):\Stab_{\iga}(\psi_{A_{\ell - \mup}})]^{-s} \\ &  + q^{-2(\lup - \ldown)s} [\bfG(\cO_\ldown) : \bfJ(\cO_\ldown)]^{-s} q^{2 (\lup -1)} \zeta_{\bfJ(\cO_\ldown)}(s), 
\end{split}
\]
where $A(\lup,m; \omega,\delta, \alpha, \beta, \gamma ) \in \mfg(\lri_{\lup})$ are of the form (\ref{form.of.char}) with $m \geq 1$, $\omega, \delta \in \pi \cO_\ell$, and  the last term on the right side corresponds to $m = \lup$.

\begin{remark}\label{rem:SL3.SU3} The construction in Theorem~\ref{thm:construction} and the formula for the zeta function $\zeta_{\bfG(\lri_\ell) | E}(s)$ are valid for $\bfG\in \{\SL_3,\SU_3\}$ {\em mutatis mutandis} by replacing the orbits representatives by the subset of traceless representatives and  the relevant subgroups by their intersection with $\SL_3$. 
    
\end{remark}

\section{Representations of the  groups $\bfJ(\lri_\ell)$ and proof of Theorem~\ref{thm:dim.reduction.J}}\label{sec:proof.E}

\subsection{A blueprint proof of Theorem~\ref{thm:dim.reduction.J}} 

The results in Section~\ref{sec:coad.orbits} prove most of the cases in Theorem~\ref{thm:dim.reduction.J}. Indeed, scalar orbits follow using twists by one-dimensional characters and induction on the level, regular orbits are covered in \cite{KOS}, and the decomposable orbits by Theorem~\ref{thm:dim.reduction.decomposable} and induction on the rank. 
To complete the proof we need to handle the nilpotent non-regular orbit and its twists. This is the content of the following theorem which also completes the missing case in the proof of Theorem~\ref{thm:local.zeta}.

\begin{theorem}\label{thm:zeta.e} Let $\bfG \in \{\GL_3,\GU_3\}$ and let $\ell \ge 2$. Then  
				\begin{equation}\label{eq:EEJJ}
				\zeta_{\bfG(\lri_\ell) | E}(s)= [\bfG(\cO_\ldown):\bfJ(\cO_{\ldown})]^{-s}q^{-2(\lup-
					\ldown)s}\zeta_{\bfJ(\lri_{\ell-1})}(s).
				\end{equation}  
The zeta function $\zeta_{\bfJ(\lri_{\ell})}(s)$ is given explicitly by 
   \begin{equation*}
\begin{split}
      	\zeta_{\JJ(\lri_{\ell})}(s) 
      &=q^{2\ell-2}\left[(q-\epsilon)^2+ (q-\epsilon)^2(q-1) \sum_{i=1}^{\ell}q^{i-1}q^{- is}  +(q^2-1)\sum_{i=0}^{\ell-1} q^{i}\left(q^{i}(q-\epsilon)\right)^{-s}\right].
	\end{split}			
    \end{equation*}
In particular, $\zeta_{\bfG(\lri_\ell) | E}(s)$ depends only on $\ell$, $q$ and $\epsilon$.				
			\end{theorem}
			
\begin{remark}\label{rem:JSL3.JSU3} Following Remark~\ref{rem:SL3.SU3}, we note that~\eqref{eq:EEJJ} remains valid for $\bfG \in \{\SL_3,\SU_3\}$. The formula for $\zeta_{\bfJ'(\lri_\ell)}$ with $\bfJ'=\bfJ \cap \SL_3$ is given by 
   \begin{equation*}
\begin{split}
      	\zeta_{\JJ'(\lri_{\ell})}(s) 
      &=q^{2\ell-2}\left[(q-\epsilon)\left(1+ (q-1) \sum_{i=1}^{\ell}q^{i-1}q^{- is}\right)  +(q+\epsilon)\iota(q,\epsilon)^2\sum_{i=0}^{\ell-1} q^{i}\left(\frac{q^{i}(q-\epsilon)}{\iota(q,\epsilon)}\right)^{-s}\right],
	\end{split}			
    \end{equation*}
where $\iota(q,\epsilon)=\gcd(q-\epsilon,3)$ is the number of cubic roots of unity in $\rf$ or in $\Ker N^{\Rf}_{\rf}$, depending on $\epsilon$ being $+1$ or $-1$, respectively.
\end{remark}

Proving Theorem~\ref{thm:zeta.e} is the core of this section. In Section \ref{sec:construction-J} we give two constructions for the irreducible representations of $\bfJ(\lri_\ell)$. The first construction, given in Section~\ref{sec:construction-1-J}, results in an explicit formula for the associated representation zeta function which depends only on the cardinality of the residue field, hence  completing the proof of Theorem~\ref{thm:local.zeta}, and the second, given in Section~\ref{sec:construction-2-J}, proves \eqref{eq:EEJJ}, hence completing the proof of Theorem~\ref{thm:dim.reduction.J}.

\subsection{The groups and their Lie algebras} Consider the matrix
\[
  \bfc=\left[\begin{matrix}  &  & 1 \\ & \delta &  \\ \delta^2& &  \end{matrix}\right], \quad \delta \in \pi\lri_\ell.
\]

The centraliser of $\bfc$ in $\Mat_3(\Lri_\ell)$ is
\[
  \Cent_{\Mat_3(\Lri_\ell)}(\bfc)= \left\{ \left[\begin{matrix} a & x & z \\ \delta y& b & y \\\delta^2z & \delta x& a \end{matrix}\right] \mid a,b,x,y,z \in \lri_\ell \right\}.
\]

 Recall that these groups for $\GL_3$ and $\GU_3$ are defined as follows: 
\[
  \begin{split}
    \bfJ_{L, \delta}(\lri_\ell)&=\Cent_{\GL_3(\lri_\ell)}(\bfc)= \left\{ \left[\begin{matrix} a & x & z \\ 
 \delta y& b & y \\\delta^2z & \delta x& a \end{matrix}\right] \mid a,b \in \lri_\ell^\times, x,y,z \in \lri_\ell \right\}, ~\text{and}\\
    \bfJ_{U, \delta}( \lri_\ell)&=\Cent_{\GU_3(\lri_\ell)}(\rho\bfc)=\Cent_{\GU_3(\lri_\ell)}(\bfc)=\Cent_{\GL_3(\Lri_\ell)}(\bfc) \cap \GU_3(\lri_\ell).
  \end{split}
\]

The corresponding Lie rings are
\[
  \begin{split}
    \mfj_{L, \delta}(\lri_\ell)&=\left\{\left[\begin{matrix} a & x & z \\  \delta y & b & y \\ \delta^2z&  \delta x & a \end{matrix}\right] \mid a,b,x,y,z \in \lri_\ell\right\},~\text{and} \\
    \mfj_{U, \delta}(\lri_\ell)&=\left\{\left[\begin{matrix} a & x & z \\  \delta y & b & y \\ \delta^2z&  \delta x & a \end{matrix}\right] \mid  a,b,z \in \gu_1(\lri_\ell), x+y^\circ=0 \right\} = \mfjL(\Lri_\ell) \cap \gu_3(\lri_\ell). 
  \end{split}
\]
Throughout this section we fix $\delta \in \pi\cO_\ell$.  
When treating these groups and Lie rings simultaneously we write either $\bfJ_{\delta}$ or even $\bfJ$ by suppressing $\delta$   (for $\bfJ_{L, \delta}$ or $\bfJ_{U, \delta}$) and either $\mfj_\delta$ or $\mfj$ (for $\mfj_{L, \delta}$ or $\mfj_{U, \delta}$). The Pontryagin dual of these Lie rings is identified with $\mfg(\lri_\ell)/\jc(\lri_\ell)^\perp$. Moreover, we can choose a convenient section in both cases, namely $\jc(\lri_\ell)^\Tr$, such that the following holds.
\begin{lemma} 
\label{lem:g-decomposition-for-j}
$\mfg(\lri_\ell)=\jc(\lri_\ell)^\Tr \oplus \jc(\lri_\ell)^\perp$.
\end{lemma}
\begin{proof}  We first check that
  \[
    \mfj_{L}^\perp=\left\{ B  \mid  \tr(AB)=0, \forall A \in \mfj_L \right\}=\left\{ \left[\begin{matrix} \alpha & \xi & \zeta\\  -\delta \eta& 0& \eta\\ -\delta^2\zeta&  -\delta \xi & -\alpha \end{matrix}\right] \mid \xi,\eta,\zeta, \alpha \in \lri_\ell \right\}.
  \]
  Indeed,
  \begin{multline*}
    \tr\left( \left[\begin{matrix} \nu_{11} & \nu_{12} & \nu_{13} \\  \nu_{21} & \nu_{22} & \nu_{23} \\ \nu_{31}&  \nu_{32} & \nu_{33} \end{matrix}\right]\!\!\!\left[\begin{matrix} a & x & z \\  \delta y & b & y \\ \delta^2z&  \delta x & a \end{matrix}\right]\right)\\=(\nu_{11}+\nu_{33})a+\nu_{22}b+(\nu_{21}+\nu_{23}\delta)x +(\nu_{32}+\nu_{12}\delta)y+(\nu_{31}+\nu_{13}\delta^2)z,
  \end{multline*}
  and this should be $0$ for every $a,b,x,y,z$, which gives the desired complement. It is now immediate to check that $\jc(\lri_\ell)^\Tr \cap \jc(\lri_\ell)^\perp=(0)$, and similarly for $\mfj_{c,U}$, by equating
  \[
    \left[\begin{matrix} a & \delta y & \delta^2z \\  x & b & \delta x \\ z&  y & a \end{matrix}\right] = \left[\begin{matrix} \alpha & \xi & \zeta\\  -\delta \eta& 0& \eta\\ -\delta^2\zeta&  -\delta \xi & -\alpha \end{matrix}\right]. \qedhere
  \]
\end{proof}

\subsection{Co-adjoint orbits} The structure of co-adjoint orbits is governed by the field case $\ell=1$, which we treat first.

\subsubsection{Co-adjoint orbits in the field case}\label{sec.coad.F} Assume now that $\ell=1$. In this case, we are looking at the action of $\JJc(\kk)$ on $\mfg(\kk)/\jc(\kk)^\perp$. We can translate this action to an action on $\jc(\kk)^\Tr$ by
\[
  g: \jc(\kk)^\Tr \to \jc(\kk)^\Tr, \quad X \mapsto X'= \mathrm{proj}_{\jc(\kk)^\Tr}\left(gXg^{-1}\right),
\]
for $g =  \left[\begin{matrix} a & x & z \\ & b & y \\ &  & a \end{matrix}\right] \in \JJc(\kk)$ and $X=X(\alpha,\beta, \sigma, \tau, \nu)$, explicitly given by

\begin{equation}\label{coad.ac.F}
  \left[\begin{matrix} \alpha &  &  \\   \sigma & \beta &  \\ \nu &  \tau & \alpha \end{matrix}\right] \! \mapsto  \!\left[\begin{array}{lll} \alpha + \frac{x}{2a}\sigma -\frac{y}{2b}\tau+\frac{xy}{2ab}\nu &  &  \\   \frac{b}{a}\sigma +\frac{y}{a}\nu& \!\!\beta-\frac{x}{a}\sigma +\frac{y}{b}\tau-\frac{xy}{ab}\nu &  \\ \nu &  \frac{a}{b}\tau -\frac{x}{b}\nu & \!\! \alpha +  \frac{x}{2a}\sigma -\frac{y}{2b}\tau+\frac{xy}{2ab}\nu\end{array}\!\!\right]\!\!.
\end{equation}

\begin{proposition}\label{prop.coadF} The coadjoint orbits of $\JJc(\rf)$ on the cross section $\jc(\kk)^\Tr$ are given as follows. Let
  \[
    X=X(\alpha,\beta,\sigma,\tau,\nu)
    = \left[\begin{matrix} \alpha &  &  \\   \sigma & \beta &  \\ \nu &  \tau & \alpha \end{matrix}\right] \in \jc(\kk)^\Tr.
  \]

  For $\bfJ=\JLc$, $\alpha,\beta,\sigma,\tau, \nu \in \rf$, and the orbits are classified by
          \medskip

          \!\!\begin{tabular}{ | l |  l  |  l  |  l  |  l |}
            \hline
            type    & \# of orbits & size     & representatives                   & invariants of the orbits                                                     \\
            \hline
            \hline
            $(i)$   & $q^2$        & $1$      & $X(\alpha,\beta,0,0,0)$           & $\alpha , \beta \in \kk$                                                     \\
            \hline
            $(ii)$  & $q^2(q-1)$   & $q^2$    & $X(\alpha,\beta,\sigma,\tau,\nu)$ & $\nu \in \rf^\times$, $2\alpha+\beta \in \rf$, $\sigma\tau-\beta\nu \in \rf$ \\
            \hline
            $(iii)$ & $q(q+1)$     & $q(q-1)$ & $X(\alpha,\beta,\sigma,\tau,0)$   & $2\alpha+\beta \in \kk$ ,  $[\sigma:\tau] \in \mathbb{P}^1(\rf)$             \\
            \hline
          \end{tabular}

          \medskip

    For $\bfJ=\JUc$, $\alpha,\beta,\sigma,\tau, \nu \in \Rf$ satisfying $\alpha, \beta, \nu \in \gu_1(\kk)$ and $\sigma+\tau^\circ=0$, and the orbits are classified by

          \medskip

          \!\!\begin{tabular}{ | l |  l  |  l  |  l  |  l |}
            \hline
            type    & \# of orbits & size     & representatives                            & invariants of the orbits                                                         \\
            \hline
            \hline
            $(i)$   & $q^2$        & $1$      & $X(\alpha,\beta,0,0,0)$                    & $\alpha , \beta \in \gu_1(\kk)$                                                  \\
            \hline
            $(ii)$  & $q^2(q-1)$   & $q^2$    & $X(\alpha,\beta,\sigma,-\sigma^\circ,\nu)$ & $0 \ne \nu, 2\alpha+\beta \in \gu_1(\rf)$, $\sigma^\circ\sigma+\beta\nu \in \rf$ \\
            \hline
            $(iii)$ & $q(q-1)$     & $q(q+1)$ & $X(\alpha,\beta,\sigma,-\sigma^\circ,0)$   & $2\alpha+\beta \in \gu_1(\kk)$,  $[\sigma:-\sigma^\circ] \in \mathbb{P}^1(\Rf)$  \\
            \hline
          \end{tabular}

          \medskip

\end{proposition}

\begin{proof}[Sketch of proof] For proof, first we compute $gXg^{-1}$ and project onto $\jc(\rf)^{\Tr}$. The projection in this case is easy: replace entries above the diagonal with zeros and replace the (1,1) and (3,3) entries by their average $(x'_{11}+x'_{33})/2$. The numerical values are computed using a count of the total number of elements of each type and the computation of the stabilisers for each type:

  \begin{enumerate}

    \item [(i)] The stabiliser is the whole group (and these are the only case with that stabiliser).

    \item[(ii)] $\nu \ne 0$; then
      \[
        \Stab_{\JJc (\rf)}(X)= \left\{ \left[\begin{matrix} a & x & z \\ & b & y \\ &  & a \end{matrix}\right] \mid  a,b,z \in \Rf, \quad \begin{matrix} x= \nu^{-1}\tau (a-b),  \\ y=\nu^{-1}\sigma(a - b) \end{matrix}\right\} \cap \bfG(\rf).
      \]

    \item[(iii)] $\nu = 0$;  if $\sigma \ne 0$, then
      \[
        \Stab_{\JJc (\kk)}(X)= \left\{ \left[\begin{matrix} a & x & z \\ & a & y \\ &  & a \end{matrix}\right] ~\mid a,y,z \in \Rf, x=\sigma^{-1}\tau y \right\} \cap \bfG(\kk).
      \]
      If $\tau \ne 0$ then we get a similar group with $x$ a free parameter and $y=\tau^{-1}\sigma x$.  \qedhere
  \end{enumerate}
\end{proof}

\subsubsection{Branching rules for the general case}

The following is an analogue of Proposition~\ref{fibre.over.e}.
\begin{proposition}[co-adjont orbits of $\JJc$ on $\jc^\vee$]

  Every $\Ad(\JJc(\lri_\ell))$-orbit $\Omega \subset \mfg(\lri_\ell)/\jc(\lri_\ell)^\perp \cong \jc(\lri_\ell)^\vee$ has a representative $W \in \jc(\Lri_\ell)^\Tr \cap \mfg(\lri_\ell)$ of the form
  \begin{equation}\label{can.form}
    W=W(\ell, m; \alpha,\beta,\sigma,\tau,\nu)
    = \underbrace{\left[\begin{matrix} \alpha& 0  &0 \\  0  & \beta & 0 \\ 2(\beta-\alpha) \delta  & 0  & \alpha \end{matrix}\right]}_{U} + \underbrace{  \left[\begin{matrix} 0 & 0 & 0 \\   \sigma & 0 & 0 \\ \nu &  \tau & 0 \end{matrix}\right]}_{V},
  \end{equation}
  with $m = \min \{\val(\sigma), \val(\nu), \val(\tau)\} $ and
  \begin{itemize}
    \item In the linear case $\alpha,\beta, \sigma,\tau,\nu \in \lri_{\ell}$;

    \item In the unitary case $\alpha+\alpha^\circ=\beta+\beta^\circ$ and $\nu+\nu^\circ=\sigma+\tau^\circ=0$.

  \end{itemize}
  Moreover, $ U$ as given in (\ref{can.form}) is $\Ad(\JJ(\lri_\ell))$-stable.

\end{proposition}

\begin{proof} The form of the matrix $W$ follows from Lemma~\ref{lem:g-decomposition-for-j}. The part $U$ is $\mathrm{Ad}(\JJ(\cO_\ell)$ follows because
\[
\left[\begin{matrix} \alpha& 0  &0 \\  0  & \beta & 0 \\ 2(\beta-\alpha) \delta  & 0  & \alpha \end{matrix}\right] = \alpha I + (\beta-\alpha) \left[\begin{matrix} 0 & 0  & 0 \\  0  & 1 & 0 \\ 2 \delta  & 0  &  0 \end{matrix}\right]
\]
and 
\[
\delta \left[\begin{matrix} 0 & 0  & 0 \\  0  & 1 & 0 \\ 2 \delta  & 0  &  0 \end{matrix}\right] = \left[\begin{matrix} 0 & 0  & 1 \\  0  & \delta & 0 \\ \delta^2  & 0  &  0 \end{matrix}\right] \mod (\mfj_L^{\perp}). 
\]

\end{proof}

The following  is an analogue of Theorem~\ref{branching.e.GL}. 

\begin{theorem}\label{branch.Jc} Let $W=W(\ell,\ell; \alpha_0,\beta_0, 0, 0, 0)  \in \mfj(\Lri_\ell)^\Tr \cap \mfg(\lri_\ell)$ be as in \eqref{can.form} with $m=\ell$. Then
  \begin{enumerate}

    \item  The fibre $\phi^{-1}_{\ell+1,\ell}(W)$ is a torsor over $\mfj(\rf)^\Tr = \mfj(\Rf)^\Tr \cap \mfg(\rf)$.

    \item The fibre $\phi^{-1}_{\ell+1,\ell}(W)$ is a $\JJ(\lri_{\ell+1})$-space and the $\JJ(\lri_{\ell+1})$-orbits are in bijection with the $\JJ(\rf)$-orbits in $\mfj(\rf)^\vee$.

  \end{enumerate}

  \begin{proof} Fix a base point $$\wt{W}=\wt{W}(\ell+1,\ell+1; \wt{\alpha}_0,\wt{\beta}_0, 0, 0, 0) \in \phi^{-1}_{\ell+1,\ell}(W)$$
    with $\wt{\alpha}_0 \equiv_{\ell} \alpha_0$ and $\wt{\beta}_0 \equiv_{\ell} \beta_0$. 
    Then every element in the fibre has the form $\wt{W}+\pi^\ell X$ for $X \in \mfj(\kk)^\Tr$.
    As $W$ is $\JJ(\lri_\ell)$-stable, $\phi^{-1}_{\ell+1,\ell}(W)$ is a $\JJ(\lri_{\ell+1})$-space. As $\wt{W}$ is $\JJ(\lri_{\ell+1})$-stable, the conjugation action  $g(\wt{W}+\pi^\ell X)g^{-1}$ translates into conjugation action $\ol{g}X\ol{g}^{-1}$ of the reduction modulo $\pi$ on the $\mfj(\rf)^\Tr$ as in \eqref{coad.ac.F}.
  \end{proof}

\end{theorem}

\subsubsection{Explicit action and stabilisers in the ring case} We repeat the ideas from \S\ref{sec.coad.F} and adapt them to the more general situation. From Maple we are getting the following.

\begin{proposition}[Stabilisers]
  \label{prop:j-stabilisers}
  The stabiliser of the character $W=W(\ell,m;\alpha,\beta,\sigma,\tau, \nu)$ from \eqref{can.form} is

  \begin{enumerate}

    \item[(i)] if $m=\ell$ the stabiliser is the whole group.

    \item[(ii)] if $m < \ell$ and $\nu \not \equiv 0 \mod (\pi)$, then
      \[
        \Stab_{\bf J(\lri_\ell)}(W)= \left\{ \left[\begin{matrix} a & x & z \\ \delta y& b & y \\ \delta^2z & \delta x& a \end{matrix}\right] ~\Bigg|~ \begin{matrix} a,b,z \in \Lri_\ell \qquad \qquad \qquad \\ x \equiv_{{\ell-m}} \tau \nu^{-1}(a - b+\delta z) \\ y \equiv_{{\ell-m}} \sigma \nu^{-1} (a - b+\delta z)\end{matrix}\right\} \cap \bfG(\lri_\ell).
      \]

    \item[(iii)] if $m<\ell$, $\nu \equiv 0 \mod(\pi)$ and $\sigma \not \equiv 0 \mod(\pi)$, then
      \[
        \Stab_{\bfJ(\lri_\ell)}(W)= \left\{ \left[\begin{matrix} a & x & z \\ \delta y& b & y \\  \delta^2z &  \delta x& a \end{matrix}\right] ~\Bigg|~ \begin{matrix} a,y,z \in \Lri_\ell  \\ b \equiv_{{\ell-m}} a+ \delta z- \sigma^{-1} \nu y \\ x \equiv_{{\ell-m}}\sigma^{-1}\tau y \end{matrix}\right\} \cap \bfG(\lri_\ell).
      \]

      \noindent if $m < \ell$, $\nu \equiv 0 \mod (\pi)$ and $\tau \not \equiv 0 \mod(\pi)$, then
      \[
        \Stab_{\bfJ(\lri_\ell)}(W)= \left\{ \left[\begin{matrix} a & x & z \\  \delta y& b & y \\ \delta^2z &  \delta x& a \end{matrix}\right] ~\Bigg|~ \begin{matrix} a,x,z \in \Lri_\ell   \\ b \equiv_{{\ell-m}} a+ \delta z- \tau^{-1}\nu  x \\ y \equiv_{{\ell-m}}\tau^{-1}\sigma x  \end{matrix}\right\} \cap \bfG(\lri_\ell).
      \]

  \end{enumerate}

\end{proposition}

\subsection{Construction of the irreducible representations}\label{sec:construction-J}
In this section we construct the irreducible representations of $\JJ(\lri_\ell)$ in two different to ways. The first construction is analogous to the construction of regular characters of $\bfG(\lri_\ell)$ as given in ~\cite{KOS}. The second construction is analogous to the non-regular irreducible characters of $\bfG(\lri_\ell)$ as given in Section~\ref{sec:construction-of-non-regular-characters-G}. We will use both constructions to describe the non-regular part of the representation zeta function of~$\bfG(\lri_\ell)$.

\subsubsection{Construction~1}\label{sec:construction-1-J}

Consider the smallest congruence subgroup $\JJ(\lri_\ell)^{\ell-1} \cong (\mfj(\lri_1),+)$ of the group~$\JJ(\lri_\ell)$.
The characters of $\JJ(\lri_\ell)^{\ell-1}$ are of the form $\psi_{\bar{X}}$ with $\bar{X}=\bar{X}(\bar{\alpha},\bar{\beta},\bar{\sigma},\bar{\tau},\bar{\nu})$. Their $\JJ(\lri_\ell)$-orbits are given explicitly in
Proposition~\ref{prop.coadF}. 

We first consider the case of $\bsigma = \btau = \bnu =0$.
By Lemma~\ref{lem:comm-subgp-J} and  Proposition~\ref{prop:comm.sbgp-extension}, we have $[\JJ(\cO_\ell), \JJ(\cO_\ell)] \cap \JJ(\cO_\ell)^{\ell-1} \subseteq \Ker(\psi_{X(\bar{\alpha}, \bar{\beta}, 0,0,0)})$, and therefore $\psi_{X(\bar{\alpha}, \bar{\beta}, 0,0,0)}$ extends to $\JJ(\cO_\ell)$ for every $\bar{\alpha}, \bar{\beta} \in \rf$. By Clifford theory, $\mathrm{Irr}(\JJ(\cO_\ell) \mid \psi_{X(\bar{\alpha}, \bar{\beta}, 0,0,0)}) $ are obtained from the irreducible representations of $\JJ(\cO_{\ell-1})$. 

We now focus on the remaining cases. Let $\bX = \bX(\balpha,\bbeta, \bsigma, \btau, \bnu) \in (\mfj(\cO_1)^\Tr \cap \mfg(\cO_1))$ such that  $(\bsigma, \btau, \bnu) \neq (0,0,0)$. Let $X = X(\ell, m; \alpha,\beta, \sigma, \tau, \nu) \in \mfg(\cO_\ell)$ be a lift of $ \bX=\bX(\balpha,\bbeta, \bsigma, \btau, \bnu)$ . Define $\psi_{X}: \JJ(\cO_\ell)^\lup \rightarrow \mathbb C^\times$ by 
\[
\psi_X(\mathrm{exp} (\pi^{\lup} Y) ) = \psi (\pi^{\lup} \tr(XY)).   
\]
The map $\psi_X$ is a well defined character of $\JJ(\cO_\ell)^\lup$ that extends $\psi_{\bX}$. The construction now follows similar lines as the construction of regular characters of $\bfG(\cO_\ell)$ in \cite{KOS}: 
\begin{itemize}
  \item[(a)] $I_{\JJ (\cO_\ell)}(\psi_X) = \mathrm{Stab}_{\JJ(\cO_\ell)}(X) \JJ(\cO_\ell)^\ldown  $, where $\mathrm{Stab}_{\JJ(\cO_\ell)}(X)$ is given in Proposition~\ref{prop:j-stabilisers}. 
  \item[(b)]  For $(\bar{\sigma}, \bar{\tau}, \bar{\nu}) \neq (0,0,0)$, the group $\Stab_{\JJ(\cO_\ell)}(X)$ is abelian modulo $\pi^{\ldown}$.
  \item[(c)] For $\ell$ odd: The radical of the bilinear form $$\beta^J_X: \frac{\JJ (\cO_\ell)^{\ldown}}{\JJ (\cO_\ell)^{\lup}} \times \frac{\JJ (\cO_\ell)^{\ldown}}{\JJ (\cO_\ell)^{\lup}} \rightarrow \mathbb C^\times; \,\, \beta_X(A, B) = \psi(\tr(X[A,B])$$
    is $\frac{\Stab_{\JJ(\cO_\ell)}(X)^{\ldown} \JJ(\cO_\ell)^\lup}{\JJ (\cO_\ell)^{\lup}}.$
\end{itemize}
We omit the details of this construction and directly give the number of orbits and order of the stabilisers for both even and odd cases. 

\begin{itemize}

  \item $\ell=2r$.
  
\noindent When $\nu \ne 0$ the number of orbits is $q^2(q-1)q^{3(r-1)}$, the cardinality of stabiliser modulo $\JJ(\lri_\ell)^r$ is $q^{3(r-1)} (q-\epsilon)^2q$, and the dimension of the corresponding representations is~$q^\ell$. 

\medskip

        \noindent When $\nu = 0$ the number of orbits is $q(q+\epsilon)q^{3(r-1)}$, the cardinality of the stabiliser modulo $\JJ(\lri_\ell)^r$  is $q^{3(r-1)} (q-\epsilon)q^2$, and the dimension of the relevant representation is $\left(q^{\ell-1}(q-\epsilon)\right)^{-s}$.

\bigskip

  \item $\ell=2r+1$.

        \noindent When $\nu \ne 0$ the number of orbits is $q^2(q-1)q^{3(r-1)}$, the cardinality of the stabiliser modulo $\JJ(\lri_\ell)^{r}$ is $q^{3(r-1)} (q-\epsilon)^2q$. There are $q^3$ extensions to the radical, the induction from a Lagrangian subspace gives a dimension shift by $q$, and the dimension of the corresponding representations is~$q^\ell$. 

\medskip

       \noindent  When $\nu = 0$ the number of orbits is $q(q+\epsilon)q^{3(r-1)}$, the cardinality of the stabiliser modulo $\JJ(\lri_\ell)^r$ is $q^{3(r-1)} (q-\epsilon)q^2$. There are $q^3$ extensions to the radical, the induction from a Lagrangian subspace gives a dimension shift by $q$, and the dimension of the corresponding representations is~$\left(q^{\ell-1}(q-\epsilon)\right)^{-s}$.

\end{itemize}

Combining these with Clifford theory gives a recursive expression for the representation zeta function of $\JJ(\cO_\ell)$:
\[
  \zeta_{\JJc(\lri_{\ell})}(s)= q^2\zeta_{\JJc(\lri_{\ell-1})}(s)+\underbrace{q^{3\ell}(1-\epsilon q^{-1})^2(1-q^{-1}) q^{-\ell s}}_{\nu \ne 0} +\underbrace{q^{3\ell-1}(1-q^{-2})\left(q^{\ell-1}(q-\epsilon)\right)^{-s}}_{\nu = 0}.
\]

Using the latter we can write zeta function explicitly.

\begin{proposition}
\label{prop:zeta-function-J}

   The representation zeta function of $\JJ_\delta(\cO_\ell)$ is given by
  \[
    \zeta_{\JJ_\delta(\kk)}(s)=(q-\epsilon)^2 + (q^2-1)(q-\epsilon)^{-s}+(q-1)(q-\epsilon)^2q^{-s},
  \]
 for $\ell=1$, and for $\ell \geq 2$ by
 \[
    \begin{split}
      \zeta_{\JJ_\delta(\lri_{\ell})}(s) 
      &=q^{2\ell-2}(q-\epsilon)^2+ (q-\epsilon)^2(q-1) \sum_{i=1}^{\ell}q^{i+2\ell-3}q^{- is} +(q^2-1)\sum_{i=0}^{\ell-1} q^{i+2\ell-2}\left(q^{i}(q-\epsilon)\right)^{-s}.
    \end{split}
  \]
 
\end{proposition}
We note that the groups $\bfJ_{L, \delta}(\cO_\ell)$ for $\ell \geq 2$ are not necessarily isomorphic for different values of $\delta$. However, we obtain the following result regarding their group algebras:    
\begin{remark} 
For any $\delta_1, \delta_2 \in \pi \cO_\ell$, we have  
\[
\C[\bfJ_{L, \delta_1}(\cO_\ell) ] \cong \C[\bfJ_{L, \delta_2}(\cO_\ell) ] \,\, \mathrm{and} \,\, \C [\bfJ_{U, \delta_1}(\cO_\ell) ] \cong \C[\bfJ_{U, \delta_2}(\cO_\ell) ]. 
\]
\end{remark}

\subsubsection{Construction~2}
\label{sec:construction-2-J}
We now give an alternate method to construct irreducible representations of $\JJ(\cO_\ell).$ For this, we work with the representations lying above characters of the maximal abelian principal congruence group $\JJ(\lri_\ell)^{\lup}$, where $\lup = \lceil \ell/2 \rceil$. Recall that $\ldown = \lfloor \ell/2 \rfloor.$

Let $W \in \mfj(\Lri_\ell)^\Tr \cap \mfg(\lri_\ell)$ be as in  \eqref{can.form}, that is 
\begin{equation*}
   W = W(\ell, m; \alpha,\beta,\sigma,\tau,\nu)
    = \underbrace{\left[\begin{matrix} \alpha& 0  &0 \\  0  & \beta & 0 \\ 2(\beta-\alpha) \delta  & 0  & \alpha \end{matrix}\right]}_{U} + \underbrace{  \left[\begin{matrix} 0 & 0 & 0 \\   \sigma & 0 & 0 \\ \nu &  \tau & 0 \end{matrix}\right]}_{V}.
\end{equation*}
Every character $\JJ(\cO_\ell)^\lup$ is of the form $\psi_{W_\ldown}.$ Up to multiplication by a central character, we assume that 
\[
W_{\ldown} = \left[\begin{matrix} 0& 0  &0 \\  \sigma  & \omega & 0 \\ \nu + 2\omega  \delta  & \tau   & 0 \end{matrix}\right] \mod \pi^{\ldown}. 
\]

We define the extension of $\psi_{W_\ldown}$, denoted $\psi_{W_\lup}$,  to $\JJ(\cO_\ell)^\ldown$ by using the exponential map.
\[
  \psi_{W_{\lup}}(g)=\psi\left(\tr(W_{\lup} \log(g)) \right).
\]
Note that this is a  well-defined character, and here $W_{\lup}$  is an arbitrary lift of $W_{\ldown}$ with $W$ chosen appropriately.   
Recall $m = \min\{\val(\sigma), \val(\nu), \val(\tau)\}$ and 
\[
  \mdown =\lfloor (\ell-m)/2 \rfloor \quad \text{and} \quad \mup =\lceil (\ell-m)/2 \rceil. 
\]
We consider the case of $m < \lup$ and $m \geq \lup$ separately. 
\medskip

\paragraph{\bf The case $m \geq \lup$:} 
\begin{lemma}
\label{lem:extension-for-m-l2-case} The character $\psi_{W_\lup}$ extends to a linear character $\psi_{W_{\ell - \mup}}: \JJ(\lri_\ell) \to \mathbb C^\times$.
\end{lemma}

\begin{proof} From Proposition~\ref{prop:comm.sbgp-extension}, we obtain
\[
\psi_{W_\lup}([\JJ(\cO_\ell),\JJ(\cO_\ell)] \cap \JJ(\cO_\ell)^\ldown) = 1.   
\]
Hence $\psi_{W_\lup}$ extends to $\JJ(\cO_\ell).$ 
\end{proof}
Thus every irreducible character of $\JJ(\cO_\ell)$ lying above $\psi_{W_\lup}$ is obtained from $ \psi_{W_{\ell - \mup}}$ tensored with an irreducible representation of $\JJ(\cO_\ell)$ modulo $\pi^{\ldown}.$ 

\medskip

\paragraph{\bf The case $m < \lup$:} 

\begin{lemma} The character $\psi_{W_\lup}$ extends to a linear character $\psi_{W_{\ell - \mup}}: \bfJ(\lri_\ell)^{\mup} \to \mathbb C^\times$.

\end{lemma}
\begin{proof} For this case, we first separately extend the characters $\psi_U$ and $\psi_V$ of $\JJ(\cO_\ell)^\ldown$.  The character $\psi_U$ extends to $\JJ(\cO_\ell)$, denoted $\wt{\phi_U}$,  by the arguments as given in Lemma~\ref{lem:extension-for-m-l2-case}. We now extend $\psi_V$. We note that $2\mup \ge \mdown + \mup = \ell-m$, and therefore
  \[
    [\bfJ(\lri_\ell)^{\mup},\bfJ(\lri_\ell)^{\mup}] \leq \bfJ(\lri_\ell)^{\ell-m} \leq \ker(\psi_{V}). \qedhere
  \]

We define the extension of $\psi_V$ using the exponential map.
\[
  \psi_{V_{\ell - \mup}} \left(\exp \pi^{\mup}X\right)=\psi\left(\pi^{\mup}\tr(XV) \right).
\]
Note that $2\mup+m \ge \ell$ and therefore the implicit use of the logarithm on the right hand side truncates after the linear term and we get a well-defined character. Let
\[
  \psi_{W_{\ell - \mup}}=\wt{\psi_U} \otimes \psi_{V_{\ell - \mup}}:  \bfJ(\lri_\ell)^{\mup} \to \mathbf C^\times.
\]
\end{proof}
To finish the construction of the representations, we look at the following sequence of groups: 
\[
  \bfJ(\lri_\ell) \supset \Stab_{\bfJ(\lri_\ell)}(\psi_{W_{\ldown}}) \supset \Stab_{\bfJ(\lri_\ell)}(\psi_{W_{\ell - \mup}}) \supset  \bfJ(\lri_\ell)^{\mdown}\supset \bfJ(\lri_\ell)^{\mup} \supset  \bfJ(\lri_\ell)^{\lup}.
\]

The following can be obtained from the $\psi_{W_{\ell - \mup}}$  definition. 
\begin{enumerate}

  \item $\Stab_{\bfJ(\lri_\ell)}(\psi_{W_{\ell - \mup}}) =\Stab_{\bfJ(\lri_\ell)}(W)\bfJ(\lri_\ell)^{\mdown}$

  \item $\Stab_{\bfJ(\lri_\ell)}(\psi_{W_{\ell - \mup}})/\bfJ(\lri_\ell)^{\mdown}$ is abelian,

\end{enumerate}
where $\Stab_{\bfJ(\lri_\ell)}(W)$ is as given in the Proposition~\ref{prop:j-stabilisers}. The construction here is exactly on the lines of non-regular characters of $\bfG(\cO_\ell)$ as described in Section~\ref{sec:construction-of-non-regular-characters-G}, so we leave the rest of the details.

\subsection{Proof of Theorem~\ref{thm:zeta.e}}

The following expression for $\zeta_{\JJ(\lri_{\ell})}(s)$, which follows from the second construction given in Section~\ref{sec:construction-2-J}, is analogous to the expression for $\zeta_{\bfG(\lri_\ell) | E}(s)$ given in Section \ref{eq:explicit.zeta.E}.
\[
  \begin{split}
    \zeta_{\JJ(\lri_{\ell})}(s)= \sum_{m=0}^{\ldown} &\sum_{W(\lup,m;\alpha, \beta, \sigma, \nu, \tau) \in \mfj(\lri_{\lup})^\Tr} [\JJ(\lri_\ell):\Stab_{\JJ(\lri_\ell)}(\psi_{W_{\ldown}})]^{-1} \\& \times \sum_{\wt{\phi_W} | W} [\Stab_{\JJ(\lri_\ell)}(\psi_{W_{\ldown}}):\Stab_{\JJ(\lri_\ell)}(\psi_{W_{\ell-\mup}})]^{-1}q^{-2(\mup-\mdown)}q^{-(\mup-\mdown)s} \\
    & \times [\Stab_{\JJ(\lri_\ell)}(\psi_{W_{\ell-\mup}}): \bfJ(\lri_\ell)^{\ldown}]  [\JJ(\lri_\ell):\Stab_{\JJ(\lri_\ell)}(\psi_{W_{\ell-\mup}})]^{-s}
    \\ &  +  q^{2 (\lup -1)} \zeta_{\bfJ(\cO_\ldown)}(s).  
  \end{split}
\]
Here the last term  corresponds to $m = \lup$. We can now write $\zeta_{\JJ(\lri_{\ell})}(s)$ as a sum of two zeta functions 
\[
\zeta_{\JJ(\lri_{\ell})}(s)=\zeta_{\JJ(\lri_{\ell})|\Xi}(s)+\zeta_{\JJ(\lri_{\ell})|\Xi'}(s),
\]
with $\Xi = \{ W(\lup, m; \alpha, \beta, \sigma, \nu, \tau) \in \mfj(\lri_{\lup})^\Tr \mid  \alpha, \beta \in \pi\cO_\ell, m \geq 1 \}$ and $\Xi' = \mfj(\lri_{\lup})^\Tr \smallsetminus \Xi$. Comparing $\zeta_{\bfG(\lri_\ell) | E}(s)$ with $\zeta_{\bfJ(\lri_\ell)|\Xi} (s)$, we observe that 
\[
\zeta_{\bfG(\lri_\ell) | E}(s) = [\bfG(\cO_\ldown) : \bfJ(\cO_\ldown)]^{-s} q^{-2(\lup - \ldown) s}  \zeta_{\bfJ(\lri_\ell)\mid_{\Xi}} (s). 
\]

Finally, noting  that $\zeta_{\bfJ(\lri_\ell) \mid_{\Xi}} (s) = \zeta_{\bfJ(\lri_{\ell -1})} (s)$, we get the equality \eqref{eq:EEJJ}.  That, together with Proposition~
\ref{prop:zeta-function-J}, finishes the proof of the theorem.

\bibliographystyle{siam}
\bibliography{refs}

\end{document}